\documentclass[12pt]{article}
\usepackage{amsfonts}
\usepackage{mathrsfs}
\usepackage{amsthm}
\usepackage{amsmath}
\usepackage{amssymb}
\usepackage{latexsym}
\usepackage{graphicx}
\usepackage{subfigure}

\usepackage{color,varioref}
\usepackage{longtable}
\usepackage{rotating,multirow,array}
\usepackage{enumerate}
\usepackage{caption}
\usepackage{float}
\definecolor{red}{rgb}{0.9,0,0}

\definecolor{green}{rgb}{0,0.9,0}

\definecolor{blue}{rgb}{0,0,0.9}

\def\mc{\multicolumn}
\def\cA{{\cal A}} 
\def\cB{{\cal B}}    \def\cF{{\cal F}}

\def\cP{{\cal P}}

\def\cT{{\cal T}}

\def\pobj{{\mbox{\tt pobj}}}
\def\dobj{{\mbox{\tt dobj}}}

\newcommand{\inprod}[2]{\langle #1 , #2 \rangle}

\def\norm#1{\|#1\|}

\def\cI{{\cal I}}

\def\inprod#1#2{\langle#1, \, #2\rangle}

\def\ni{\noindent}

\def\diag{{\rm diag}}

\def\nn{\nonumber}

\newtheorem{theorem}{Theorem}[section]
\newtheorem{proposition}{Proposition}[section]

\newtheorem{definition}{Definition}[section]
\newtheorem{assumption}{Assumption}[section]

\begin{document}

\title{\bf A dual semismooth Newton based augmented Lagrangian method for large-scale linearly constrained sparse group square-root Lasso problems}

\author{
Chengjing Wang\thanks{School of Mathematics, Southwest Jiaotong University, No.999, Xian Road, West Park, High-tech Zone, Chengdu 611756, China ({Email: renascencewang@hotmail.com}). }
\ and Peipei Tang\thanks{{\bf Corresponding author}, School of Computer and
Computing Science, Zhejiang University City
College, Hangzhou 310015, China (Email: tangpp@zucc.edu.cn). This author's research is supported by the Natural Science Foundation of Zhejiang Province of China under Grant No. LY19A010028, the Zhejiang Science and Technology Plan Project of China (No. 2020C03091, No. 2021C01164) and the Scientific Research Foundation of Zhejiang University City College (No. X-202112).},
}

\maketitle


\begin{abstract}
Square-root Lasso problems are proven robust regression problems. Furthermore, square-root regression problems with structured sparsity also plays an important role in statistics and machine learning.
In this paper, we focus on the numerical computation of large-scale linearly constrained sparse group square-root Lasso problems. In order to overcome the difficulty that there are two nonsmooth terms in the objective function, we propose a dual semismooth Newton (SSN) based augmented Lagrangian method (ALM) for it. That is, we apply the ALM to the dual problem with the subproblem solved by the SSN method. To apply the SSN method, the positive definiteness of the generalized Jacobian is very important. Hence we characterize the equivalence of its positive definiteness and the constraint nondegeneracy condition of the corresponding primal problem. In numerical implementation, we fully employ the second order sparsity so that the Newton direction can be efficiently obtained. Numerical experiments demonstrate the efficiency of the proposed algorithm.
\end{abstract}


\section{Introduction}

In this paper, we consider the following linearly constrained sparse group square-root Lasso (cssLasso) problem
\begin{eqnarray}
\min_{x\in\mathbb{R}^{n}}\left\{\norm{\cA x-b}+\lambda_1\sum_{j=1}^{J}\omega_{j}\norm{x_{G_j}}+\lambda_2\norm{x}_{1}\,|\,\cB_{E} x-c_{E}=0,\,\cB_{I} x-c_{I}\in\mathbb{R}^{m_I}_{+}\right\},\label{eq:sparse group square-root Lasso}
\end{eqnarray}
where $\cA: \mathbb{R}^{n}\rightarrow\mathbb{R}^{m},\,\cB_{E}: \mathbb{R}^{n}\rightarrow\mathbb{R}^{m_E}$ and $\cB_{I}: \mathbb{R}^{n}\rightarrow\mathbb{R}^{m_I}$ are given linear mappings whose adjoints are denoted as $\cA^{*}, \cB_{E}^{*}$ and $\cB_{I}^{*}$, respectively, $b\in\mathbb{R}^{m}, c_{E}\in\mathbb{R}^{m_E}$ and $c_{I}\in\mathbb{R}^{m_I}$ are given vectors, $\omega_{j}> 0\ (j=1,2,\ldots,J)$ is a weight parameter, $G_{j}\, (j=1,2,\ldots,J)$ is an index set which contains the indices in the $j$th group, $x=(x_{G_{1}};x_{G_{2}};\ldots;x_{G_{J}})$, $\lambda_1,\lambda_2 \geq 0$ are regularization parameters, and $\mathbb{R}^{m_I}_{+}$ denotes an $m_I$-dimensional positive orthant cone. And $\norm{\cdot}$ and $\norm{\cdot}_{1}$ denote the $l_2$ norm and $l_1$ norm, respectively. Throughout the paper, we assume that $\cup_{j=1}^{J}G_{j}=\{1,2,\ldots,n\}$ and $G_{i}\cap G_{j}=\phi$ for $1\leq i < j \leq n$.

In statistics and machine learning, the Lasso model
\begin{eqnarray}
\min_{x\in\mathbb{R}^{n}}\left\{\frac{1}{2}\norm{\cA x-b}^{2}+\lambda\norm{x}_{1}\right\},\label{eq:Lasso}
\end{eqnarray}
where $\lambda>0$ is a regularization parameter, performs both variable selection and regularization in order to enhance the prediction accuracy and interpretability of the resulting statistical model. It was originally introduced in geophysics literature by Santosa and Symes \cite{SantosaS1986}, and later was independently rediscovered and popularized by Tibshirani \cite{Tibshirani1996}. Yuan and Lin \cite{YuanL2006} proposed the group Lasso model
\begin{eqnarray*}
\min_{x\in\mathbb{R}^{n}}\left\{\frac{1}{2}\norm{\cA x-b}^{2}+\lambda\sum_{j=1}^{J}\omega_{j}\norm{x_{G_j}}\right\},\label{eq:group Lasso}
\end{eqnarray*}
which is more suitable for variable selection and it can be regarded as an extension of the Lasso for selecting groups of variables. The problem of selecting grouped variables arises naturally in many practical situations with the multifactor analysis-of-variance problem. Friedman, Hastie, and Tibshirani \cite{FriedmanHT2010} considered a model with a more general penalty that blends the Lasso ($l_1$ norm) with the group Lasso ($l_2$ norm). This penalty yields solutions that are sparse at both the group and individual feature levels. This is the so called sparse group Lasso model as below
\begin{eqnarray}
\min_{x\in\mathbb{R}^{n}}\left\{\frac{1}{2}\norm{\cA x-b}^{2}+\lambda_1\sum_{j=1}^{J}\omega_{j}\norm{x_{G_j}}+\lambda_2\norm{x}_{1}\right\}.\label{eq:sparse group Lasso}
\end{eqnarray}

Although the Lasso model \eqref{eq:Lasso} is an attractive estimator, it relies on knowing the standard deviation $\varsigma$ of the noise. Estimation of $\varsigma$ is nontrivial when $n$ is large, particularly when $n\gg m$. In view of this flaw, Belloni, Chernozhukov and Wang \cite{BelloniCW2011} proposed the square-root Lasso model
\begin{eqnarray}
\min_{x\in\mathbb{R}^{n}}\Big\{\norm{\cA x-b}+\lambda\norm{x}_{1}\Big\}.\label{eq:square-root Lasso}
\end{eqnarray}
In contrast to the Lasso estimator \eqref{eq:Lasso}, the square-root Lasso estimator \eqref{eq:square-root Lasso} is independent of $\varsigma$. It has also been proved that the square-root Lasso estimator achieves the near-oracle rates of convergence under suitable design conditions. In addition, there are other aspects of significance for the square-root Lasso. The scaled Lasso proposed by Sun and Zhang \cite{SunZ2012} is essentially equivalent to the square-root Lasso \eqref{eq:square-root Lasso}. But the scaled Lasso is computationally expensive. Xu, Caramanis and Mannor \cite{XuCM} pointed out that the square-root Lasso \eqref{eq:square-root Lasso} is equivalent to a robust linear regression problem subject to an uncertainty set. It deserves mentioning that Stucky and van de Geer \cite{Sara2017} established the sharp oracle property of the square-root Lasso. Motivated by the wide applicability of group selection methods, Bunea, Lederer, and She \cite{BuneaLS} studied the group version of the square-root Lasso
\begin{eqnarray}
\min_{x\in\mathbb{R}^{n}}\left\{\norm{\cA x-b}+\lambda\sum_{j=1}^{J}\omega_{j}\norm{x_{G_j}}\right\}.\label{eq:group square-root Lasso}
\end{eqnarray}
They also showed that the group square-root Lasso estimator adapts to the unknown sparsity of the regression vector,
and has the same optimal estimation and prediction accuracy as the group Lasso estimators, under the same minimal conditions on the model.

Meanwhile, in \cite{AltenbuchingerRZSDWHGHOS,AtitallahTKAAH,GainesZ2018,JamesPR2010,LinSFL,ShiZL} and so on, the constrained Lasso was considered, which takes the following form
\begin{eqnarray*}
\min_{x\in\mathbb{R}^{n}}\left\{\left.\frac{1}{2}\norm{\cA x-b}^{2}+\lambda\norm{x}_{1}\,\right|\,\cB_{E} x-c_{E}=0,\,\cB_{I} x-c_{I}\in\mathbb{R}^{m_I}_{+}\right\}.\label{eq:constrained Lasso}
\end{eqnarray*}
It extends the widely-used Lasso to handle linear constraints allowing the user to incorporate prior information into the model.

In an almost parallel pattern, we consider the cssLasso \eqref{eq:sparse group square-root Lasso}, which may strike an effective compromise between the square-root Lasso and the group square-root Lasso, yielding sparseness at the group and individual predictor levels.
It deserves mentioning that Chu, Toh and Zhang \cite{ChuTZ2021} recently considered the square-root regression problems without any constraints.
Yang and Xu \cite{YangX} characterize the equivalence of the cssLasso \eqref{eq:sparse group square-root Lasso} and the robust regression problem under a specially chosen uncertainty set. That is, they provide an interpretation of the cssLasso from a robustness perspective.

In recent years, great process has been made in the computation of the Lasso type problems. Li, Sun and Toh \cite{LiSunToh2018} proposed a semismooth Newton augmented Lagrangian method (SSN-ALM) to solve the Lasso problem \eqref{eq:Lasso}. Since problem \eqref{eq:Lasso} is piecewise linear-quadratic, by the conclusion in Sun's PhD thesis \cite{Sun1986} the subdifferential of the corresponding Lagrangian function $\cT_l$ is piecewise polyhedral, then by Robinson's work \cite{Robinson81}, the calmness condition holds for $\cT_l^{-1}$, which is equivalent to the metric subregularity condition holds for $\cT_l$ according to \cite[Theorem 3H.3]{DontchevR2009}. So based on the conclusion presented by Luque \cite{Luque}, the arbitrarily linear convergence rate can be guaranteed. In fact, the paper \cite{LiSunToh2018} goes beyond the Lasso model. The authors proved that $\cT_l$ is metrically subregular under the second order sufficient condition of the dual problem for a relatively more generalized model problem. Zhang et al. \cite{ZhangZST} also adopted an efficient Hessian based algorithm for solving large-scale sparse group Lasso problems \eqref{eq:sparse group Lasso}, which is essentially the SSN-ALM. Based on the local error bound condition established in \cite[Theorem 1]{ZhangJL}, they presented the metric subregularity condition for the primal problem. So the arbitrarily linear convergence rate can be guaranteed if applying the ALM on the dual problem.

As for the computation of the group square-root Lasso \eqref{eq:group square-root Lasso} and the more generalized cssLasso \eqref{eq:sparse group square-root Lasso}, it is more challenging mainly because the loss function is also nonsmooth in addition to the regularizer. We have to deal with the structured problem with two nonsmooth terms in computation. Although the primal alternating direction method of multipliers (pADMM) is adopted in \cite{LiZYL2015}, the efficiency of the pADMM is still not satisfying especially for some large-scale real data problems. So it is necessary to redesign an efficient algorithm for the cssLasso problem \eqref{eq:sparse group square-root Lasso}. Since the ALM is an ideal approach for the Lasso type problems, it is certainly a competitive candidate for the square-root Lasso type problems.
Besides, when applying the ALM, we need to solve the subproblems as accurately and efficiently as possible. The SSN method is usually an ideal approach to solve the subproblems since the sparse structure of the generalized Jacobian can be fully employed to greatly reduce the computational cost. But we must guarantee that the Jacobian of the subproblem is positive definite before using the SSN method. The positive definiteness of the Jacobian is not usually obvious, so we need to employ other equivalent conditions to characterize it. That is, we hope to prove the equivalence of the primal nondegeneracy condition and the positive definiteness of the Jacobian of the subproblem.

The remaining parts of this paper are organized as follows. In Section \ref{sec:Preliminaries}, we
introduce some basic knowledge about the proximal mapping and the error bound condition, which plays a key role in the analysis of the convergence rate of our algorithm. 
In Section \ref{sec:Generalized Jacobian}, we describe the detailed structures of several generalized Jacobians. In Section \ref{sec:ALM}, we present the details of the algorithm. In Section \ref{sec:Theoretical results}, we present some theoretical results which include the characterization of the equivalent conditions, the strong semismoothness of the involved function and the convergence results. In Section \ref{sec:Numerical issues for solving the subproblem}, we discuss the numerical issues about how to efficiently solve the Newton direction. In Section \ref{sec:Numerical experiments}, we present the numerical results to demonstrate the efficiency of our algorithm. Finally, we give some concluding remarks in Section \ref{sec:Conclusion}.

\subsection{Additional notations}
\label{subsec:Additional notations}
Let $\mathcal{X}$ and $\mathcal{Y}$ be two real finite dimensional Euclidean spaces which are equipped with the inner product $\langle\cdot,\cdot\rangle$. We denote $x\circ y$ as the Hadamard product of two given vectors $x,y\in\mathcal{X}$. For a given vector $x$, $\textrm{supp}(x)$ denotes the support of $x$, i.e.,
the set of indices such that $x_{i}\neq 0$. For any convex function $p: S\subset\mathcal{X} \rightarrow (-\infty,\infty]$, its conjugate function is denoted by $p^*$, i.e., $p^*(x)=\sup_{y}\{\langle x,y \rangle-p(y)\}$. For a given closed convex set $C$ and a vector $x$, we denote the distance from $x$ to $C$ by $\text{dist}(x,C) := \inf_{y\in C}{\|x-y\|}$ and the Euclidean projection of $x$ onto $C$ by $\Pi_{C}(x) := \mathop{\mbox{argmin}}_{y\in C}{\|x-y\|}$, the interior of $C$ is defined as $\mbox{int}(C)$ and its boundary as $\partial(C)$. The set $T_{C}(z)$ denotes the tangent cone to the set $C$ at the point $z$, and $\textrm{lin}(C):=C\cap (-C)$ denotes the lineality space of $C$. For any set-valued mapping $F: \mathcal{X} \rightrightarrows \mathcal{Y}$, $\text{gph}\ F$ denotes the graph of $F$, i.e., $\text{gph}\ F:=\{(x,y)\in \mathcal{X}\times \mathcal{Y}\ |\ y \in F(x) \}$. In addition to $\mathbb{R}^{m_I}_{+}$, we use $\mathbb{R}^{m_I}_{-}$ to denote an $m_I$-dimensional negative orthant cone. We use $\mathbb{R}^{m_I}_{++}$ and $\mathbb{R}^{m_I}_{--}$ to denote two sets whose elements are $m_I$-dimensional vectors with all their entries positive and negative, respectively. We also use $\mathbb{S}^{m_I}$ to denote $m_I$-dimensional symmetric matrix space. For $\lambda>0$ and $r$ a positive integer, we define $B_{q,r}^{\lambda} := \{x\in\mathbb{R}^{r}\,|\,\norm{x}_{q}\leq \lambda\}$, where $q=1,2,\infty$.

\section{Preliminaries}
\label{sec:Preliminaries}
In this section, we discuss some properties of the convex composite optimization problem. It is very important for the local convergence rate of the proposed algorithm in this paper.

Given a closed proper convex function $f:\mathcal{X}\rightarrow(-\infty,+\infty]$, the proximal mapping $\mbox{Prox}_{f}(\cdot)$ associated with $f$ is defined by
\begin{eqnarray*}
\mbox{Prox}_{f}(x):=\mathop{\mbox{argmin}}_{u\in\mathcal{X}}\left\{f(x)+\frac{1}{2}\|u-x\|^{2}\right\},\quad \forall\, x\in\mathcal{X}.
\end{eqnarray*}
For any $x\in\mbox{dom}(f)$, by Moreau's Theorem (see, e.g., Theorem 31.5 of \cite{Rockafellar70}), we have $\mbox{Prox}_{tf}(x)+t\mbox{Prox}_{f^{*}/t}(x/t)=x$ with a given parameter $t>0$. The epigraph of $f$ is defined as the set
\begin{eqnarray*}
\mbox{epi}f:=\Big\{(x,c)\in \textrm{dom}(f)\times\mathbb{R}\ \Big|\ f(x)\leq c\Big\}.
\end{eqnarray*}

Let $\cT:\mathcal{X}\rightrightarrows\mathcal{Y}$ be a multifunction. It is called a monotone operator if
\begin{eqnarray*}
\langle x-x',w-w'\rangle\geq0,\quad\mbox{whenever}\ w\in \cT(x),\ w'\in \cT(x').
\end{eqnarray*}
The monotone operator is said to be maximal monotone if, in addition, the graph
\begin{eqnarray*}
\textrm{gph}\,\cT:=\Big\{(x,y)\in\mathcal{X}\times\mathcal{Y}\,|\,y\in \cT(x) \Big\}
\end{eqnarray*}
is not properly contained in the graph of any other monotone operator $\cT':\mathcal{X}\rightarrow\mathcal{Y}$.

In the work \cite{Rockafellar76a}, Rockafellar studied a fundamental proximal point algorithm to solve
\begin{eqnarray*}
0\in \cT(x),
\end{eqnarray*}
with $\cT$ being a maximal monotone operator. That is, for an arbitrary initial point $x^{0}$, a sequence $\{x^{k}\}$ is generated by the following rule
\begin{eqnarray*}
x^{k+1}\approx (\mathcal{I}+\sigma_{k}\cT)^{-1}(x^{k}).
\end{eqnarray*}

For the sake of later convenience, we also present the definition of the error bound condition, which is essentially important for establishing the local convergence rate of the proximal point algorithm in \cite{Rockafellar76a}.
\begin{definition}
Let $F:\mathcal{X}\rightrightarrows\mathcal{Y}$ be a multifunction and $y$ satisfy $F^{-1}(y)\neq\emptyset$. If there exists $\varepsilon>0$ such that
\begin{eqnarray}
\mbox{dist}(x,F^{-1}(y))\leq \kappa\mbox{dist}(y,F(x)),\quad  \forall\, x\in\mathcal{X}\ \mbox{such that}\ \mbox{dist}(y,F(x))\leq\varepsilon,\label{eq:error bound}
\end{eqnarray}
is valid, then $F$ is said to satisfy the error bound condition at the point $y$ with modulus $\kappa$.
\label{def:error bound}
\end{definition}

For the convenience of the statement, we denote $h(x)=\norm{x}$ and $p(x)=p_1(x)+p_2(x)$, where $p_1(x)=\lambda_1\sum_{j=1}^{J}\omega_{j}\norm{x_{G_j}}$, $p_2(x)=\lambda_2\norm{x}_{1}$, then we can write problem \eqref{eq:sparse group square-root Lasso} in the following form
\begin{eqnarray}
\min_{x\in\mathbb{R}^{n}}\left\{h(\mathcal{A}x-b)+p(x)\,\Big|\,\cB_{E} x-c_{E}=0,\,\cB_{I} x-c_{I}\in\mathbb{R}^{m_I}_{+}\right\}.
\label{eq:sparse group square-root Lasso2}
\end{eqnarray}
 By introducing slack variables $y$ and $z$, we can write problem \eqref{eq:sparse group square-root Lasso2} in an equivalent form
\begin{eqnarray*}
(P) & \max\limits_{(x,y,z)\in \mathbb{R}^{n}\times\mathbb{R}^{m}\times\mathbb{R}^{m_I}}\Big\{-(h(y)+p(x))\ \Big|\,\mathcal{A}x-y=b,\,\cB_{E} x-c_{E}=0,\\
& \cB_{I} x-c_{I}+z=0,\,z\in\mathbb{R}^{m_I}_{-}\Big\}.
\end{eqnarray*}
The dual of $(P)$ takes the following form
\begin{eqnarray*}
(D) & \min\limits_{u,w\in \mathbb{R}^{m},s\in\mathbb{R}^{n},v_{E}\in\mathbb{R}^{m_E},v_{I},\hat{v}_{I}\in\mathbb{R}^{m_I}}\Big\{h^{*}(w)+p^{*}(s)+\inprod{b}{u}+\inprod{c_{E}}{v_{E}}+\inprod{c_{I}}{v_{I}} \Big|\\
&\mathcal{A}^{*}u + \cB^{*}_{E} v_{E} + \cB^{*}_{I} v_{I} + s=0,\, -u+w=0,\,v_{I}-\hat{v}_{I}=0,\,\hat{v}_{I}\in\mathbb{R}^{m_I}_{-}\Big\}.
\end{eqnarray*}
The Lagrangian function for the dual problem $(D)$ is
\begin{eqnarray}\label{lagrangian function}
l(u,v_{E},v_{I},\hat{v}_{I},w,s;x,y,z) &=& h^{*}(w)+p^{*}(s)+\inprod{b}{u}+\inprod{c_{E}}{v_{E}}+\inprod{c_{I}}{v_{I}}\nn\\
&&-\inprod{x}{\mathcal{A}^{*}u+\cB^{*}_{E} v_{E}+\cB^{*}_{I} v_{I}+s}-\inprod{y}{-u+w}\,-\inprod{z}{v_{I}-\hat{v}_{I}},\nn\\
&&\hat{v}_{I}\in\mathbb{R}^{m_I}_{-}.
\end{eqnarray}
The KKT condition associated with the dual problem $(D)$ is given as follows
\begin{eqnarray}
&\mathcal{A}x-y=b,\,\cB_{E} x-c_{E}=0,\,\cB_{I} x-c_{I}+z=0,\nn\\
&\mathcal{A}^{*}u + \cB^{*}_{E} v_{E} + \cB^{*}_{I} v_{I} + s=0,\,-u+w=0,\,v_{I}-\hat{v}_{I}=0,\nn\\
&0\in\partial h^{*}(w)-y,\ 0\in\partial p^{*}(s)-x,\nn\\
&\hat{v}_{I}\in\mathbb{R}^{m_I}_{-},\,\cB_{I} x-c_{I}\in\mathbb{R}^{m_I}_{+},\,\inprod{\hat{v}_{I}}{\cB_{I} x-c_{I}}=0.\label{eq:KKT}
\end{eqnarray}


Throughout the paper we assume the following condition holds.

\begin{assumption}
In problem $(P)$, for the optimal solution point $(\bar{x},\bar{y},\bar{z})$, $\bar{y}\neq 0$.
\label{assumption:y != 0}
\end{assumption}


For the problem we consider in this paper, we define a function $f:\,\mathbb{R}^{n}\times\mathbb{R}^{m}\times\mathbb{R}^{m_I}\rightarrow\overline{\mathbb{R}}$ as below
\begin{eqnarray*}
f(x,y,z):=h(y)+p(x)+\delta_{\{0\}}(\mathcal{A}x-b)+\delta_{\{0\}}(\cB_{E} x-c_{E})+\delta_{\{0\}}(\cB_{I} x-c_{I}+z)+\delta_{\mathbb{R}^{m_I}_{-}}(z),
\end{eqnarray*}
then we can define the operators $\mathcal{T}_{f}$ and $\mathcal{T}_{l}$ related to the closed proper convex function $f$ in $(P)$ and the convex-concave function $l$ in (\ref{lagrangian function}), respectively by
\begin{eqnarray*}
\begin{aligned}
&\mathcal{T}_{f}(x,y,z):=\partial f(x,y,z),\quad \mathcal{T}_{l}(u,v_{E},v_{I},w,s,x,y,z):=\Big\{(u',v_{E}',v_{I}',w',s',x',y',z')\,|\,(u',v_{E}',\\
&v_{I}',w',s',-x',-y',-z')\in \partial l(u,v_{E},v_{I},w,s,x,y,z)\Big\},
\end{aligned}
\end{eqnarray*}
with their inverses given by
\begin{eqnarray*}
\begin{aligned}
&\mathcal{T}_{f}^{-1}(x',y',z'):=\partial f^{*}(x',y',z'),\quad \mathcal{T}_{l}^{-1}(u',v_{E}',v_{I}',w',s',x',y',z'):=\Big\{(u,v_{E},v_{I},w,s,x,y,z)\\
&\,|\,(u',v_{E}',v_{I}',w',s',-x',-y',-z')\in \partial l(u,v_{E},v_{I},w,s,x,y,z)\Big\}.
\end{aligned}
\end{eqnarray*}

Note that $\mathcal{T}_{f}$ is a maximal monotone operator in $\mathbb{R}^{n}\times\mathbb{R}^{m}\times\mathbb{R}^{m_I}$ (see e.g., \cite{Minty64,Moreau65}) and $\mathcal{T}_{l}$ is also a maximal monotone operator in $\mathbb{R}^{m}\times\mathbb{R}^{m_E}\times\mathbb{R}^{m_I}\times\mathbb{R}^{m}\times\mathbb{R}^{n}\times\mathbb{R}^{n}\times\mathbb{R}^{m}\times\mathbb{R}^{m_I}$ due to Corollary 37.5.2 of \cite{Rockafellar70}.

\section{The generalized Jacobians of $\mbox{Prox}_{\sigma h}(\cdot), \mbox{Prox}_{\sigma p}(\cdot)$ and $\Pi_{\mathbb{R}^{m_I}_{+}}(\cdot)$}
\label{sec:Generalized Jacobian}
For the convenience of the subsequent statement, given $\sigma>0$, we need to characterize the structures of the Clarke generalized Jacobians of $\mbox{Prox}_{\sigma h}(\cdot), \mbox{Prox}_{\sigma p}(\cdot)$ and $\Pi_{\mathbb{R}^{m_I}_{+}}(\cdot)$, respectively.

\begin{itemize}
\item $\mbox{Prox}_{\sigma h}(\cdot)$:

For any $u_{h}\in \mathbb{R}^{m}$,
\begin{eqnarray*}
\Pi_{B_{2,m}^{\sigma}}(u_{h}) &=& \left\{\begin{array}{ll}\frac{\sigma u_{h}}{\norm{u_{h}}},&\mbox{if}\,\|u_{h}\|>\sigma,\\
u_{h}, &\mbox{otherwise}.\end{array}\right.
\end{eqnarray*}
Then
\begin{eqnarray*}
\mbox{Prox}_{\sigma h}(u_{h}) = u_{h} - \Pi_{B_{2,m}^{\sigma}}(u_{h}).
\end{eqnarray*}
We can calculate
\begin{eqnarray*}
\partial\Pi_{B_{2,m}^{\sigma}}(u_{h}) &=& \left\{\begin{array}{ll}\left\{\frac{\sigma}{\norm{u_{h}}}(I-\frac{u_{h}u_{h}^{T}}{\norm{u_{h}}^{2}})\right\},&\mbox{if}\ \norm{u_{h}}>\sigma,\\
\left\{I-\frac{t}{\sigma^{2}} u_{h}u_{h}^{T}\,|\,0\leq t\leq 1\right\},&\mbox{if}\ \norm{u_{h}}=\sigma,\\
\{I\}, &\mbox{if}\ \norm{u_{h}}<\sigma.\end{array}\right.
\end{eqnarray*}
Hence the Clarke generalized Jacobian of $\mbox{Prox}_{\sigma h}(\cdot)$ at $u_{h}$ is
\begin{eqnarray*}
\partial\mbox{Prox}_{\sigma h}(u_{h}) = I - \partial\Pi_{B_{2,m}^{\sigma}}(u_{h}).
\end{eqnarray*}

\item $\mbox{Prox}_{\sigma p}(\cdot)$:

In \cite{ZhangZST}, Zhang et al. have characterized the so-called surrogate generalized Jacobian of $\mbox{Prox}_{\sigma p}(\cdot)$. In order to make the paper self-contained, we also describe it here.

We define a linear operator $\cP:=(\cP_{1},\cP_{2},\ldots,\cP_{J}):\,\mathbb{R}^{n}\rightarrow \mathbb{R}^{n}$, where $\cP_{j}:\,\mathbb{R}^{n}\rightarrow \mathbb{R}^{|G_{j}|}$ is defined by $\cP_{j}x:=x_{G_j}$, $j=1,\ldots,J$. We also define $B_{2}:=B_{2,|G_{1}|}^{\sigma\lambda_{1,1}}\times B_{2,|G_{2}|}^{\sigma\lambda_{1,2}}\times\cdots\times B_{2,|G_{J}|}^{\sigma\lambda_{1,J}}$. For any $u_{p}\in \mathbb{R}^{n}$, by Theorem 4 in \cite{Yu2013}, we have
\begin{eqnarray}
\mbox{Prox}_{\sigma p}(u_{p}) = v - \Pi_{B_{2}}(v),
\label{eq:prox-sig*p}
\end{eqnarray}
where
\begin{eqnarray}
v=\mbox{Prox}_{\sigma p_2}(u_{p})= u_{p} - \Pi_{B_{\infty,n}^{\lambda_2}}(u_{p})=\mbox{sign}(u_{p})\circ \max(|u_{p}|-\sigma\lambda_{2},0),
\label{eq:prox-sig*p2}
\end{eqnarray}
and
\begin{eqnarray*}
\Pi_{B_{2}}(v):=\left(\begin{array}{c} \Pi_{B_{2,|G_{1}|}^{\sigma\lambda_{1,1}}}(\cP_{1}v)\\ \vdots\\ \Pi_{B_{2,|G_{J}|}^{\sigma\lambda_{1,J}}}(\cP_{J}v)\end{array}\right).
\end{eqnarray*}

Furthermore, for any $u_{p}\in\mathcal{R}^{n}$ we are ready to define an alternative for the Clarke generalized Jacobian of $\mbox{Prox}_{\sigma p}(u_{p})$ as below
\begin{eqnarray}
\hat{\partial}\mbox{Prox}_{\sigma p}(u_{p}) &:=& \Big\{(I-\cP^{*}\Sigma\cP)\Theta\,\Big|\,\Sigma=\textrm{Diag}(\Sigma_1,\ldots,\Sigma_J),\Sigma_j\in \partial\Pi_{B_{2,|G_{j}|}^{\sigma\lambda_{1,j}}}(\cP_{j}v),\nn\\
&& j=1,\ldots,J,\,v=\mbox{Prox}_{\sigma p_2}(u_{p}),\,\Theta\in\partial\mbox{Prox}_{\sigma p_2}(u_{p})\Big\}, \label{eq:partial Prox_sigma*p}
\end{eqnarray}
where
\begin{eqnarray*}
\Sigma_j &\in& \left\{\begin{array}{ll}\left\{\frac{\sigma\lambda_{1,j}}{\norm{\cP_{j}v}}(I-\frac{(\cP_{j}v)(\cP_{j}v)^{T}}{\norm{\cP_{j}v}^{2}})\right\},&\mbox{if}\,\norm{\cP_{j}v}>\sigma\lambda_{1,j},\\
\left\{I-\frac{t}{(\sigma\lambda_{1,j})^{2}} (\cP_{j}v)(\cP_{j}v)^{T}\,|\,0\leq t\leq 1\right\},&\mbox{if}\,\norm{\cP_{j}v}=\sigma\lambda_{1,j},\\
\{I\}, &\mbox{if}\,\norm{\cP_{j}v}<\sigma\lambda_{1,j},\end{array}\right.
\label{eq:Sigmaj}
\end{eqnarray*}
and
$\Theta=\textrm{Diag}(\theta)$ with
\begin{eqnarray}
\label{eq:theta}
\theta_{i} \in \left\{\begin{array}{ll}\{0\},& \mbox{if}\ |(u_{p})_i| < \sigma\lambda_{2},\\
\left[0,1\right],& \mbox{if}\ |(u_{p})_i| = \sigma\lambda_{2},\\
\{1\}, & \mbox{if}\ |(u_{p})_i| > \sigma\lambda_{2}.\end{array}\right.
\end{eqnarray}

In \cite[Theorem 3.1]{ZhangZST}, Zhang et al. have proved that for any $M\in \hat{\partial}\mbox{Prox}_{\sigma p}(u_{p})$, $M$ is symmetric and positive semidefinite.

\item $\Pi_{\mathbb{R}^{m_I}_{+}}(\cdot)$:

For any $u_{r}\in\mathbb{R}^{m_I}$,
\begin{eqnarray*}
\Pi_{\mathbb{R}^{m_I}_{+}}(u_{r})=\frac{1}{2}(u_{r}+|u_{r}|).
\end{eqnarray*}
Then the Clarke generalized Jacobian of $\Pi_{\mathbb{R}^{m_I}_{+}}(\cdot)$ can be described as
\begin{eqnarray}
\partial\Pi_{\mathbb{R}^{m_I}_{+}}(u_{r}) = \left\{V\in\mathbb{S}^{{m_I}}\,\left|\,V= \textrm{Diag}(\tilde{v}),\,(\tilde{v})_{i}\in\left\{\begin{array}{ll}\{1\},&\mbox{if}\ (u_{r})_{i}>0,\\
\left[0,1\right], & \mbox{if}\ (u_{r})_{i}=0,\\
\{0\}, &\mbox{if}\ (u_{r})_{i}<0.\end{array}\right.\right.\right\}.
\label{eq:Jacobian-Projector-R+}
\end{eqnarray}
\end{itemize}

\section{The SSN-ALM for the dual problem (D)}
\label{sec:ALM}
In this section, we give a brief introduction of the SSN-ALM for the dual problem (D).
For a given $\sigma>0$, the augmented Lagrangian function associated with the dual problem $(D)$ is defined by
\begin{eqnarray*}
L_{\sigma}(u,v_{E},v_{I},w,s;x,y,z)&:=&h^{*}(w)+p^{*}(s)+\inprod{b}{u}+\inprod{c_{E}}{v_{E}}+\inprod{c_{I}}{v_{I}}\\
&&+\frac{\sigma}{2}\|\mathcal{A}^{*}u+\cB^{*}_{E}v_{E}+\cB^{*}_{I}v_{I}+s-\sigma^{-1}x\|^{2}+\frac{\sigma}{2}\|w-u-\sigma^{-1}y\|^{2}\\
&&+\frac{1}{2\sigma}\|\Pi_{\mathbb{R}^{m_I}_{+}}(\sigma v_{I}-z)\|^{2}-\frac{1}{2\sigma}(\|x\|^{2}+\|y\|^{2}+\|z\|^{2}).
\end{eqnarray*}
For each inner subproblem of the ALM, the SSN method will be applied to obtain an inexact solution.

\bigskip
\ni\fbox{\parbox{\textwidth}{\noindent{\bf Algorithm SSN-ALM:} Let $\sigma_{0}>0$ be a given parameter. Select an initial point $(u^{0},v_{E}^{0},v_{I}^{0},w^{0},s^{0},x^{0},y^{0},z^{0})\in \mathbb{R}^{m}\times\mathbb{R}^{m_E}\times\mathbb{R}^{m_I}\times\mathbb{R}^{m}\times\mathbb{R}^{n}\times \mathbb{R}^{n}\times\mathbb{R}^{m}\times\mathbb{R}^{m_I}$. For $k=0,1,\ldots$, iterate the following steps.
\begin{description}
\item [Step 1.] Apply Algorithm SSN to compute
\begin{eqnarray}\label{eq:subproblem of alm}
(u^{k+1},v_{E}^{k+1},v_{I}^{k+1},w^{k+1},s^{k+1})&\approx&\mathop{\mbox{argmin}}_{u,v_{E},v_{I},w,s}\Big\{\Phi_{k}(u,v_{E},v_{I},w,s)\nn\\
&& :=L_{\sigma_k}(u,v_{E},v_{I},w,s;x^{k},y^{k},z^{k})\Big\}.
\end{eqnarray}

\item [Step 2.] Compute
\begin{eqnarray*}
\left\{\begin{array}{l}
x^{k+1}=x^{k}-\sigma_{k}(\mathcal{A}^{*}u^{k+1}+\cB_{E}^{*}v_{E}^{k+1}+\cB_{I}^{*}v_{I}^{k+1}+s^{k+1}),\\
y^{k+1}=y^{k}-\sigma_{k}(w^{k+1}-u^{k+1}),\\
z^{k+1}=-\Pi_{\mathbb{R}^{m_I}_{+}}(\sigma_k v_{I}^{k+1}-z^{k}).
\end{array}\right.
\end{eqnarray*}
Update $\sigma_{k+1}=\rho\sigma_{k}$ for some $\rho\geq 1$.
\end{description}}}

\bigskip

In the following we go to the details of how to solve the subproblem \eqref{eq:subproblem of alm}. For simplicity, we omit the superscript or subscript $k$.
Given $(x,y,z)\in \mathbb{R}^{n}\times\mathbb{R}^{m}\times\mathbb{R}^{m_I}$ and $\sigma>0$, for any $(u,v_{E},v_{I})\in \mathbb{R}^{m}\times\mathbb{R}^{m_E}\times\mathbb{R}^{m_I}$, we define
\begin{eqnarray*}
\varphi(u,v_{E},v_{I}) &:=& \inf_{w\in \mathbb{R}^{m},s\in\mathbb{R}^{n}}\Phi(u,v_{E},v_{I},w,s)=h^{*}(\textrm{Prox}_{h^{*}/\sigma}(\sigma^{-1}y+u))  \\
&&+\frac{1}{2\sigma}\norm{\textrm{Prox}_{\sigma h}(y+\sigma u)}^{2} +  p^{*}(\textrm{Prox}_{p^{*}/\sigma}(\sigma^{-1}x-\cA^{*}u-\cB_{E}^{*}v_{E}-\cB_{I}^{*}v_{I}))
\\
&&+\frac{1}{2\sigma}\norm{\textrm{Prox}_{\sigma p}(x-\sigma
(\cA^{*}u+\cB_{E}^{*}v_{E}+\cB_{I}^{*}v_{I}))}^{2}+\frac{1}{2\sigma}\|\Pi_{\mathbb{R}^{m_I}_{+}}(\sigma v_{I}-z)\|^{2}\\
&&+\inprod{b}{u}+\inprod{c_E}{v_E}+\inprod{c_I}{v_I}- \frac{1}{2\sigma}(\norm{x}^{2}+\norm{y}^{2}+\norm{z}^{2}).
\end{eqnarray*}
By Danskin's Theorem \cite{Danskin1966}, $\varphi$ is a continuously differentiable function with its gradient being Lipschitz continuous. We need to compute the solution $(\bar{u},\bar{v}_{E},\bar{v}_{I})$ of the following nonlinear system of equations
\begin{eqnarray*}
\nabla\varphi(u,v_{E},v_{I})=\left(\begin{array}{c}\textrm{Prox}_{\sigma h}(y+\sigma u) - \cA
\textrm{Prox}_{\sigma p}(x-\sigma (\cA^{*}u+\cB_{E}^{*}v_{E}+\cB_{I}^{*}v_{I}))+b\\-\cB_{E}\textrm{Prox}_{\sigma p}(x-\sigma (\cA^{*}u+\cB_{E}^{*}v_{E}+\cB_{I}^{*}v_{I}))+c_E\\-\cB_{I}\textrm{Prox}_{\sigma p}(x-\sigma (\cA^{*}u+\cB_{E}^{*}v_{E}+\cB_{I}^{*}v_{I}))+\Pi_{\mathbb{R}^{m_I}_{+}}(\sigma v_{I}-z)+c_I\end{array}\right)=0.
\end{eqnarray*}
Then we can obtain that $\bar{s}=\mbox{Prox}_{p^{*}/\sigma}(\sigma^{-1}x-\cA^{*}\bar{u}-\cB_E^{*}\bar{v}_E-\cB_I^{*}\bar{v}_I)$ and $\bar{w}=\mbox{Prox}_{h^{*}/\sigma}(\sigma^{-1}y+\bar{u})$.
Since $\mbox{Prox}_{\sigma h}(\cdot)$ and $\mbox{Prox}_{\sigma p}(\cdot)$ are both Lipschitz continuous functions, we define
\begin{eqnarray*}
\hat{\partial}^{2}\varphi(u,v_{E},v_{I}) &:=& \sigma\left(\begin{array}{ccc}\partial\textrm{Prox}_{\sigma h}(y+\sigma u)& & \\ &0& \\ & &\partial\Pi_{\mathbb{R}^{m_I}_{+}}(\sigma v_{I}-z)\end{array}\right)
\\
&&+\sigma \left(\begin{array}{c}\cA\\\cB_E\\ \cB_I\end{array}\right) \hat{\partial}\textrm{Prox}_{\sigma p}(x-\sigma (\cA^{*}u+\cB_{E}^{*}v_{E}+\cB_{I}^{*}v_{I}))\left(\begin{array}{ccc}\cA^{*} & \cB_E^{*} & \cB_I^{*}\end{array}\right),
\end{eqnarray*}
where $\partial\textrm{Prox}_{\sigma h}(y+\sigma u), \partial\Pi_{\mathbb{R}^{m_I}_{+}}(\sigma v_{I}-z)$ and
$\hat{\partial}\textrm{Prox}_{\sigma p}(x-\sigma (\cA^{*}u+\cB_{E}^{*}v_{E}+\cB_{I}^{*}v_{I}))$ are the generalized Jacobians of $\textrm{Prox}_{\sigma h}(\cdot), \Pi_{\mathbb{R}^{m_I}_{+}}(\cdot)$ and
$\textrm{Prox}_{\sigma p}(\cdot)$ at $y+\sigma u, \sigma v_{I}-z$ and $x-\sigma (\cA^{*}u+\cB_{E}^{*}v_{E}+\cB_{I}^{*}v_{I})$, respectively.
It is known from \cite{HiriartSN1984} that
\begin{eqnarray*}
\partial^{2}\varphi(u,v_{E},v_{I})(d)=\hat{\partial}^{2}\varphi(u,v_{E},v_{I})(d),\quad
\forall\, d:=\left(\begin{array}{c} d_{1}\\ d_{2}\\ d_{3}\end{array}\right)\in \mathbb{R}^{m+m_{E}+m_{I}},
\end{eqnarray*}
where $\partial^{2}\varphi(u,v_{E},v_{I})$ is the generalized
Hessian of $\varphi$ at $(u,v_{E},v_{I})$. Let $V_1\in\partial\mbox{Prox}_{\sigma h}(y+\sigma u), V_2\in\hat{\partial}\mbox{Prox}_{\sigma p}(x-\sigma(\cA^{*}u+\cB_{E}^{*}v_{E}+\cB_{I}^{*}v_{I}))$ and $V_3\in\partial\Pi_{\mathbb{R}^{m_I}_{+}}(\sigma v_{I}-z)$. Then we
define
\begin{eqnarray}
H &:=& \sigma\left(\begin{array}{ccc} V_1& & \\ &0& \\ & & V_3\end{array}\right)
+ \sigma \left(\begin{array}{c}\cA\\\cB_E\\ \cB_I\end{array}\right) V_2 \left(\begin{array}{ccc}\cA^{*} & \cB_E^{*} & \cB_I^{*}\end{array}\right)\in\hat{\partial}^{2}\varphi(u,v_{E},v_{I}).\label{eq:H-Jacobian}
\end{eqnarray}

If we can guarantee that every element $H\in\hat{\partial}^{2}\varphi(u,v_{E},v_{I})$ is positive definite, then we can apply the
SSN method to get an approximate solution of $(\bar{u},\bar{v}_{E},\bar{v}_{I})$. We discuss this topic in Section \ref{sec:Theoretical results}.

Now we present the detailed algorithm of the SSN as below.

\bigskip

\ni\fbox{\parbox{\textwidth}{\noindent{\bf Algorithm SSN:
\label{alg:Newton-CG1}} Given $(x,y,z)\in \mathbb{R}^{n}\times\mathbb{R}^{m}\times\mathbb{R}^{m_I}$, $\sigma>0$,
$\mu\in(0,\frac{1}{2}), \ \overline{\eta}\in(0,1), \tau\in(0,1],
\nu_{1},\nu_{2}\in(0,1), \textrm{
  and }
  \delta\in(0,1)$, choose $(u^{0},v_{E}^{0},v_{I}^{0},w^{0},s^{0})\in \mathbb{R}^{m}\times\mathbb{R}^{m_E}\times\mathbb{R}^{m_I}\times\mathbb{R}^{m}\times\mathbb{R}^{n}$. Set $j=0$ and iterate the
following steps.
\begin{description}
\item [Step 1.] Choose $H^{j}$ in the form of \eqref{eq:H-Jacobian} with $V_1\in\partial\mbox{Prox}_{\sigma h}(y+\sigma u^{j}), V_2\in\hat{\partial}\mbox{Prox}_{\sigma p}(x-\sigma(\cA^{*}u^{j}+\cB_{E}^{*}v_{E}^{j}+\cB_{I}^{*}v_{I}^{j}))$ and $V_3\in\partial\Pi_{\mathbb{R}^{m_I}_{+}}(\sigma v_{I}^{j}-z)$.
Find the exact solution $(\Delta u^{j},\Delta v_E^{j},\Delta v_I^{j})$ or apply the preconditioned conjugate gradient (PCG) method to find an
approximate solution $(\Delta u^{j},\Delta v_E^{j},\Delta v_I^{j})$ to
\begin{eqnarray*}
(H^{j}+\varepsilon_{j}I)(\Delta u,\Delta v_E,\Delta v_I)=-\nabla\varphi(u^{j},v_E^{j},v_I^{j}),
\end{eqnarray*}
such that
\begin{eqnarray*}
\|H^{j}(\Delta u^{j},\Delta v_E^{j},\Delta v_I^{j})+\nabla\varphi(u^{j},v_E^{j},v_I^{j})\|\leq
\eta_{j}:=\min(\overline{\eta},\|\nabla\varphi(u^{j},v_E^{j},v_I^{j})\|^{1+\tau}),\label{ineq:stopping
criterion 1}
\end{eqnarray*}
where
$\varepsilon_{j}:=\nu_{1}\min\left\{\nu_{2},\|\nabla\varphi(u^{j},v_E^{j},v_I^{j})\|\right\}$.
\item [Step 2.]  Set $\alpha_{j}=\delta^{m_{j}}$, where $m_{j}$ is the first nonnegative integer $m$ for which
\begin{eqnarray*}
\varphi((u^{j},v_E^{j},v_I^{j})+\delta^{m}(\Delta u^{j},\Delta v_E^{j},\Delta v_I^{j}))&\leq&
\varphi(u^{j},v_E^{j},v_I^{j})\\&&+\mu\delta^{m}\langle\nabla
\varphi(u^{j}),(\Delta u^{j},\Delta v_E^{j},\Delta v_I^{j})\rangle.
\end{eqnarray*}
\item [Step 3.]  Set $(u^{j+1},v_E^{j+1},v_I^{j+1})=(u^{j},v_E^{j},v_I^{j})+\alpha_{j}(\Delta u^{j},\Delta v_E^{j},\Delta v_I^{j})$.
\end{description}}}

\bigskip

\section{Theoretical results}
\label{sec:Theoretical results}

In this section we shall present some theoretical results related to the SSN-ALM. Firstly, we must guarantee that the generalized Jacobian is positive definite when using the SSN method. However, the positive definiteness of the generalized Jacobian is very essential, so we shall establish an equivalent condition to characterize the positive definiteness in Subsection \ref{subsec:Perturbation analysis results}. Secondly, we analyze the strong semismoothness of the involved function in the inner problem and the local convergence rate of the SSN method in Subsection \ref{subsec:Local convergence rate of the SSN}. Thirdly, we need to analyze the global convergence and local convergence rate of the ALM in Subsection \ref{subsec:Convergence analysis of the ALM}.

\subsection{Constraint nondegeneracy and the positive definiteness of $\hat{\partial}^{2}\varphi(\bar{u},\bar{v}_{E},\bar{v}_{I})$}
\label{subsec:Perturbation analysis results}

By introducing a variable $\eta\in\mathbb{R}$, the primal problem $(P)$ can be reformulated as
\begin{eqnarray*}
(P') & \max\limits_{x\in\mathbb{R}^{n},y\in\mathbb{R}^{m},z\in\mathbb{R}^{m_I},\eta\in\mathbb{R}}\Big\{-(h(y)+\eta)\ \Big|\,\mathcal{A}x-y=b,\,\cB_{E} x-c_{E}=0,\,\cB_{I} x-c_{I}+z=0,\\
& z\in\mathbb{R}^{m_I}_{-},\,(x,\eta)\in \textrm{epi}p\Big\}.
\label{primal problem P'}
\end{eqnarray*}
The dual of problem $(P')$ is
\begin{eqnarray*}
(D')& \min\limits_{u,w\in \mathbb{R}^{m},s\in\mathbb{R}^{n},v_{E}\in\mathbb{R}^{m_E},v_{I},\hat{v}_{I}\in\mathbb{R}^{m_I},\xi\in\mathbb{R}}\Big\{h^{*}(w)+\xi+\inprod{b}{u}+\inprod{c_E}{v_E}+\inprod{c_I}{v_I}\, \Big|\\
& \cA^{*}u + \cB_{E}^{*}v_{E} + \cB_{I}^{*}v_{I} + s=0,\,-u+w=0,\,v_{I}-\hat{v}_{I}=0,\,\hat{v}_{I}\in\mathbb{R}^{m_I}_{-},\,(s,\xi)\in \textrm{epi}p^{*}\Big\}.
\label{dual problem D'}
\end{eqnarray*}

We say that $(x',y',z',\eta')$ is a feasible solution to problem $(P')$ if
\begin{eqnarray*}
(x',y',z',\eta')\in\cF_{P} &:=& \Big\{(x,y,z,\eta)\in \mathbb{R}^{n}\times \mathbb{R}^{m}\times \mathbb{R}^{m_I}\times \mathbb{R}\,\Big|\,\mathcal{A}x-y=b,\cB_{E} x-c_{E}=0,\\
&&\cB_{I} x-c_{I}+z=0,\,z\in\mathbb{R}^{m_I}_{-},\,(x,\eta)\in \textrm{epi}p\Big\},
\label{eq:primal feasible set}
\end{eqnarray*}
and $(u',v_{E}',v_{I}',\hat{v}_{I}',w',s',\xi')$ is a feasible solution to problem $(D')$ if
\begin{eqnarray*}
(u',v_{E}',v_{I}',\hat{v}_{I}',w',s',\xi')\in\cF_{D} &:=& \Big\{(u,v_{E},v_{I},\hat{v}_{I},w,s,\xi)\in \mathbb{R}^{m}\times\mathbb{R}^{m_{E}}\times\mathbb{R}^{m_{I}}\times\mathbb{R}^{m_{I}}\times \mathbb{R}^{m}\times\\
&& \mathbb{R}^{n}\times \mathbb{R}\,\Big|\, w\in\mbox{dom}h^{*},\
\cA^{*}u + \cB_{E}^{*}v_{E} + \cB_{I}^{*}v_{I} + s=0,\\
&&-u+w=0,\,v_{I}-\hat{v}_{I}=0,\,\hat{v}_{I}\in\mathbb{R}^{m_I}_{-},\,(s,\xi)\in \textrm{epi}p^{*}\Big\}.
\label{eq:dual feasible set}
\end{eqnarray*}

In order to guarantee the existence of the primal and dual solutions, we assume the following two conditions hold.
\begin{assumption}
Problem $(P')$ satisfies the condition:
\begin{eqnarray*}
&\exists\, (x^{0},y^{0},z^{0},\eta^{0})\in\mathbb{R}^{n}\times\mathbb{R}^{m}\times\mathbb{R}^{m_I}\times\mathbb{R},\textrm{ such that } \mathcal{A}x^{0}-y^{0}=b,&\\
&\cB_{E}x^{0}-c_{E}=0,\,\cB_{I} x^{0}-c_{I}+z^{0}=0,\,z^{0}\in\mathbb{R}^{m_I}_{--},\,(x^{0},\eta^{0})\in \textrm{int}(\textrm{epi}p).&
\end{eqnarray*}
\label{assumption:Slater-P}
\end{assumption}
\begin{assumption}
Problem $(D')$ satisfies the condition:
\begin{eqnarray*}
&\exists\, (u^{0},v_{E}^{0},v_{I}^{0},\hat{v}_{I}^{0},w^{0},s^{0},\xi^{0})\in\mathbb{R}^{m}\times\mathbb{R}^{m_E}\times\mathbb{R}^{m_I}\times\mathbb{R}^{m_I}\times\mathbb{R}^{m}\times\mathbb{R}^{n}\times\mathbb{R},&\\ &\textrm{such that } w^{0}\in\mbox{int}(\mbox{dom}h^{*}),\,
\mathcal{A}^{*}u^{0} + \cB^{*}_{E} v_{E}^{0} + \cB^{*}_{I} v_{I}^{0} + s^{0}=0,\, -u^{0}+w^{0}=0,\,v_{I}^{0}-\hat{v}_{I}^{0}=0,&\\
&\hat{v}_{I}^{0}\in\mathbb{R}^{m_I}_{--},\,(s^{0},\xi^{0})\in \textrm{int}(\textrm{epi}p^{*}).&
\end{eqnarray*}
\label{assumption:Slater-D}
\end{assumption}

Since Assumption \ref{assumption:y != 0} holds, the constraint nondegeneracy condition for problem $(P')$ at the primal solution $(\bar{x},\bar{y},\bar{z},\bar{\eta})$ takes the following form
\begin{eqnarray}
\cA \mathbb{R}^{n} - \mathbb{R}^{m} &=& \mathbb{R}^{m},\nn\\
\cB_{E} \mathbb{R}^{n} &=& \mathbb{R}^{m_E},\nn\\
\cB_{I} \mathbb{R}^{n} + \mathbb{R}^{m_I} &=& \mathbb{R}^{m_I},\nn\\
\mathbb{R}^{m_I} + \textrm{lin}(T_{\mathbb{R}^{m_I}_{-}}(\bar{z})) &=& \mathbb{R}^{m_I},\nn\\
(\cI\quad \cI)(\mathbb{R}^{n}\times\mathbb{R}) + \textrm{lin}(T_{\textrm{epi}p}(\bar{x},\bar{\eta})) &=& \mathbb{R}^{n}\times\mathbb{R}.
\label{eq:constraint nondegeneracy condition}
\end{eqnarray}
For any $z\in\mathbb{R}^{m_I}$,
\begin{eqnarray*}
T_{\mathbb{R}^{m_I}_{-}}(z)=\Big\{\hat{d}\in\mathbb{R}^{m_I}\,\Big|\,\hat{d}_{i}\leq 0,\textrm{ if }z_i=0;\,\,\hat{d}_{i}\in\mathbb{R},\textrm{ if }z_i<0\Big\}
\end{eqnarray*}
with its lineality space as
\begin{eqnarray*}
\textrm{lin}(T_{\mathbb{R}^{m_I}_{-}}(z))=\Big\{\hat{d}\in\mathbb{R}^{m_I}\,\Big|\,\hat{d}_{i}=0,\textrm{ if }z_i=0;\,\,\hat{d}_{i}\in\mathbb{R},\textrm{ if }z_i<0\Big\}.\label{eq:lineality of ineq constr}
\end{eqnarray*}
For any $(x,\eta)\in\mathbb{R}^{n}\times\mathbb{R}$,
\begin{eqnarray*}
T_{\textrm{epi}p}(x,\eta)=\left\{\begin{array}{ll}\mathbb{R}^{n}\times\mathbb{R},& \mbox{if}\quad (x,\eta)\in\mbox{int}(\textrm{epi}p),\\
\textrm{epi}p,& \mbox{if}\quad (x,\eta)=(0,0),\\
\left\{\left.(dx,d\eta)\in\mathbb{R}^{n}\times\mathbb{R}\ \right|\ p'(x;dx)-d\eta\leq 0\right\},&\mbox{if}\quad(x,\eta)\in\partial (\textrm{epi}p)\setminus \{(0,0)\},\end{array}\right.
\end{eqnarray*}
where $p'(x;dx)$ is the directional derivative of $p$ at $x$ in the direction $dx$.
The lineality space of $T_{\textrm{epi}p}(x,\eta)$ is
\begin{eqnarray*}
&&\mbox{lin}(T_{\textrm{epi}p}(x,\eta))\\
&=&T_{\textrm{epi}p}(x,\eta)\cap -T_{\textrm{epi}p}(x,\eta)\\
&=&\left\{\begin{array}{ll}\mathbb{R}^{n}\times\mathbb{R},&\mbox{if}\quad(x,\eta)\in\mbox{int}(\textrm{epi}p),\\
\{(0,0)\},& \mbox{if}\quad (x,\eta)=(0,0),\\
\{(dx,d\eta)\in\mathbb{R}^{n}\times\mathbb{R}\,|\, p'(x;dx)\leq d\eta\leq -p'(x;-dx)\},&\mbox{if}\quad(x,\eta)\in\partial (\textrm{epi}p)\setminus \{(0,0)\}.\end{array}\right.
\end{eqnarray*}
Since $p(x)=\lambda_1\sum\limits_{j=1}^{J}\omega_{j}\|x_{G_j}\|+\lambda_2\|x\|_{1}$, we have that
\begin{eqnarray*}
p'(x;dx)&=&\lambda_1\left(\sum_{j:x_{G_j}\neq 0}\omega_{j}\left(\frac{x_{G_j}}{\|x_{G_j}\|}\right)^{T}(dx)_{G_j}+\sum_{j:x_{G_j}=0}\omega_{j}\|(dx)_{G_j}\|\right)\\
&&+\lambda_2\left(\sum_{i:x_{i}\neq 0}\mbox{sign}(x_{i})(dx)_{i}+\sum_{i:x_{i}=0}|(dx)_{i}|\right).
\end{eqnarray*}
Therefore, we can describe the lineality space of $T_{\textrm{epi}p}(x,\eta)$ as
\begin{eqnarray*}
&&\mbox{lin}(T_{\textrm{epi}p}(x,\eta))\\
&=&\left\{\begin{array}{ll}\mathbb{R}^{n}\times\mathbb{R},&\mbox{if}\quad(x,\eta)\in\mbox{int}(\textrm{epi}p),\\
\{(0,0)\},& \mbox{if}\quad (x,\eta)=(0,0),\\
\{(dx,d\eta)\in\mathbb{R}^{n}\times\mathbb{R}\,|\, p'(x;dx)=-p'(x;-dx)=d\eta, &\\
  (dx)_{i}=0\textrm{ if } x_{i}=0\},&\mbox{if}\quad(x,\eta)\in\partial (\textrm{epi}p)\setminus \{(0,0)\}.\end{array}\right.
\end{eqnarray*}
Since at the solution point the constraint must be active, which means $(\bar{x},\bar{\eta})\in\partial (\textrm{epi}p)\setminus \{(0,0)\}$, we define a linear subspace $\textrm{T}^{\textrm{lin}}(x)\subseteq \mathbb{R}^{n}$ by
\begin{eqnarray}
\textrm{T}^{\textrm{lin}}(\bar{x}) &=&
\Big\{dx\in\mathbb{R}^{n}\,\Big|\,(dx)_{i}=0\textrm{ if } \bar{x}_{i}=0\Big\}.\label{eq:T^lin}
\end{eqnarray}
Then \eqref{eq:constraint nondegeneracy condition} is equivalent to
\begin{eqnarray}
\cB_{E}\textrm{T}^{\textrm{lin}}(\bar{x})&=&\mathbb{R}^{m_E},\nn\\
\cB_{I}\textrm{T}^{\textrm{lin}}(\bar{x}) + \textrm{lin}(T_{\mathbb{R}^{m_I}_{-}}(\bar{z}))&=&\mathbb{R}^{m_I}.
\label{eq:constraint nondegeneracy condition2}
\end{eqnarray}

Based on Assumption \ref{assumption:y != 0}, $\partial\textrm{Prox}_{\sigma h}(\bar{y}+\sigma \bar{u})$ is a singleton set $\{V_1\}$ with $V_{1}$ positive definite. Then we define an element $\widehat{H}\in\hat{\partial}^{2}\varphi(\bar{u},\bar{v}_{E},\bar{v}_{I})$
with
\begin{eqnarray}
\widehat{H} &:=& \sigma\left(\begin{array}{ccc} V_1& & \\ &0& \\ & & V^{0}_3\end{array}\right)
+ \sigma \left(\begin{array}{c}\cA\\\cB_E\\ \cB_I\end{array}\right) V^{0}_2 \left(\begin{array}{ccc}\cA^{*} & \cB_E^{*} & \cB_I^{*}\end{array}\right),
\end{eqnarray}
where
\begin{eqnarray}
V_{2}^{0}:=(I_{n}-\cP^{*}\Sigma\cP)\Theta,
\label{eq:V_2^0}
\end{eqnarray}
here
\begin{eqnarray*}
\Sigma&=&\textrm{Diag}(\Sigma_1,\ldots,\Sigma_J),\,j=1,\ldots,J,\\
\Sigma_j &=& \left\{\begin{array}{ll}\frac{\sigma\lambda_{1,j}}{\norm{\cP_{j}v}}(I-\frac{(\cP_{j}v)(\cP_{j}v)^{T}}{\norm{\cP_{j}v}^{2}}),&\mbox{if}\,\norm{\cP_{j}v}>\sigma\lambda_{1,j},\\
I, &\mbox{if}\,\norm{\cP_{j}v}\leq\sigma\lambda_{1,j},\end{array}\right.
\label{eq:Sigmaj}
\end{eqnarray*}
with $v=\mbox{Prox}_{\sigma p_2}(\bar{x}-\sigma (\cA^{*}\bar{u}+\cB_{E}^{*}\bar{v}_{E}+\cB_{I}^{*}\bar{v}_{I}))$,
$\Theta=\textrm{Diag}(\theta)$ with
\begin{eqnarray*}
\theta_{i} = \left\{\begin{array}{ll}1, & \mbox{if}\ |(\bar{x}-\sigma (\cA^{*}\bar{u}+\cB_{E}^{*}\bar{v}_{E}+\cB_{I}^{*}\bar{v}_{I}))_i| > \sigma\lambda_{2},\\
0,& \mbox{if}\ |(\bar{x}-\sigma (\cA^{*}\bar{u}+\cB_{E}^{*}\bar{v}_{E}+\cB_{I}^{*}\bar{v}_{I}))_i| \leq \sigma\lambda_{2},\end{array}\right.
\end{eqnarray*}
and
\begin{eqnarray}
V_{3}^{0}:=\textrm{Diag}(\tilde{v}_{3}),(\tilde{v}_{3})_{i}=\left\{\begin{array}{ll}1,&\mbox{if}\,(\sigma \bar{v}_{I}-\bar{z})_{i}>0,\\
0, &\mbox{if}\,(\sigma \bar{v}_{I}-\bar{z})_{i}\leq0.\end{array}\right.\label{eq:V_3^0}
\end{eqnarray}

Now we present the result of equivalence.


\begin{theorem}
For problems $(P')$ and $(D')$, the following conditions are equivalent:\\
(i) The primal constraint nondegenaracy condition (\ref{eq:constraint nondegeneracy condition2}) of the primal problem $(P')$ holds at $(\bar{x},\bar{y},\bar{z},\bar{\eta})$.\\
(ii) The elements of $\hat{\partial}^{2}\varphi(\bar{u},\bar{v}_{E},\bar{v}_{I})$ are all positive
definite.\\
(iii) $\widehat{H}$ is positive definite.
\end{theorem}
\begin{proof}
``(i)$\Rightarrow$(ii)": For any $V_2\in\hat{\partial}\textrm{Prox}_{\sigma p}(\bar{x}-\sigma (\cA^{*}\bar{u}+\cB_{E}^{*}\bar{v}_{E}+\cB_{I}^{*}\bar{v}_{I}))$, $V_3\in\partial\Pi_{\mathbb{R}^{m_I}_{+}}(\sigma \bar{v}_{I}-\bar{z})$ and $d_{1}\in\mathbb{R}^{m}$, $d_{2}\in\mathbb{R}^{m_{E}}$, $d_{3}\in\mathbb{R}^{m_{I}}$, $d=\left(\begin{array}{c} d_{1}\\ d_{2}\\ d_{3}\end{array}\right)$, we assume
\begin{eqnarray}
0 &=& \inprod{d}{Hd} = \inprod{d_1}{V_1d_1} + \inprod{d_3}{V_3d_3} + \inprod{\cA^*d_1+\cB^*_Ed_2+\cB^*_Id_3}{V_2(\cA^*d_1+\cB^*_Ed_2+\cB^*_Id_3)}\nn\\
&\geq& \inprod{d_1}{V_1d_1} + \inprod{V_3d_3}{V_3d_3} + \inprod{V_2(\cA^*d_1+\cB^*_Ed_2+\cB^*_Id_3)}{V_2(\cA^*d_1+\cB^*_Ed_2+\cB^*_Id_3)}\nn\\
&\geq& 0.\label{eq:dHd}
\end{eqnarray}
The first inequality in \eqref{eq:dHd} holds because all the eigenvalues of $V_2$ and $V_3$ are less than
or equal to one.

Since $V_1$ is positive definite and $V_{2}$, $V_{3}$ are positive semidefinite, we have $d_1=0$, $V_3d_3=0$ and $V_2(\cB^*_Ed_2+\cB^*_Id_3)=0$.
Due to the structure of $\partial\Pi_{\mathbb{R}^{m_I}_{+}}(\sigma \bar{v}_{I}-\bar{z})$, we can see that $V_3d_3=0$ implies
\begin{eqnarray*}
(d_3)_i=0,\textrm{ if }(\sigma\bar{v}_{I}-\bar{z})_i>0.\label{ineq:d3_i=0}
\end{eqnarray*}
Combining with the fact that
\begin{eqnarray*}
\textrm{if }\bar{z}_{i}<0,\textrm{ then }(\sigma \bar{v}_I-\bar{z})_i>0,
\end{eqnarray*}
which can be derived from the KKT condition \eqref{eq:KKT}, we obtain that
\begin{eqnarray*}
\textrm{if }\bar{z}_{i}<0,\textrm{ then }(d_3)_i=0.\label{eq:d2_i=0}
\end{eqnarray*}
Therefore, for any $\hat{d}\in\textrm{lin}(T_{\mathbb{R}^{m_I}_{-}}(\bar{z}))$,
\begin{eqnarray}
\inprod{d_3}{\hat{d}}=\sum_{i:\bar{z}_i=0}(d_3)_i\hat{d}_i + \sum_{i:\bar{z}_i<0}(d_3)_i\hat{d}_i = 0.\label{eq:d3 perp dhat}
\end{eqnarray}
That is,
\begin{eqnarray*}
d_3\in[\textrm{lin}(T_{\mathbb{R}^{m_I}_{-}}(\bar{z}))]^{\perp}.
\end{eqnarray*}

Let us denote the index set
\begin{eqnarray}
\Xi_{j} := G_{j}\cap \textrm{supp}(\bar{v})=\{i\in G_{j}\,|\,\theta_{i}=1\},\label{eq:Xi_j}
\end{eqnarray}
where $\bar{v}=\textrm{Prox}_{\sigma p_2}(\bar{x}-\sigma (\cA^{*}\bar{u}+\cB_{E}^{*}\bar{v}_{E}+\cB_{I}^{*}\bar{v}_{I}))$, $\theta_{i}$ is defined in \eqref{eq:theta}.

If $\bar{x}_{i}\neq 0$, then from
\begin{eqnarray*}
-(\cA^*\bar{u}+\cB_{E}^{*}\bar{v}_{E}+\cB_{I}^{*}\bar{v}_{I})\ =\ \bar{s}&=&\textrm{Prox}_{p^*/\sigma}(\sigma^{-1}\bar{x}-(\cA^*\bar{u}+\cB_{E}^{*}\bar{v}_{E}+\cB_{I}^{*}\bar{v}_{I}))\nn \\
&=& \Pi_{B^{\lambda_2}_{\infty,n} + B_{2}}(\sigma^{-1}\bar{x}-(\cA^*\bar{u}+\cB_{E}^{*}\bar{v}_{E}+\cB_{I}^{*}\bar{v}_{I})),
\label{eq:projector D1+D2}
\end{eqnarray*}
we have
\begin{eqnarray*}
\norm{(\sigma^{-1}\bar{x}-(\cA^*\bar{u}+\cB_{E}^{*}\bar{v}_{E}+\cB_{I}^{*}\bar{v}_{I}))_{G_j}} > \lambda_{1,j}  &\textrm{and}& |(\sigma^{-1}\bar{x}-(\cA^*\bar{u}+\cB_{E}^{*}\bar{v}_{E}+\cB_{I}^{*}\bar{v}_{I}))_{i}| > \lambda_2.
\end{eqnarray*}
That is, $i\in \Xi_{j}$. So $(V_2)_{\overline{\cI}\ \overline{\cI}}(\cB^*_Ed_2+\cB^*_Id_3)_{\overline{\cI}}=0$, where $\overline{\cI}:=\{i\,|\,\bar{x}_{i}\neq 0\}$, implies $(\cB^*_Ed_2+\cB^*_Id_3)_{\overline{\cI}}=0$ by the structure of the generalized Jacobian $\hat{\partial}\mbox{Prox}_{\sigma p}(\cdot)$ in \eqref{eq:partial Prox_sigma*p}.

If $\bar{x}_{i} = 0$, by \eqref{eq:T^lin} we obtain that $(dx)_{i}=0$.

Hence,
\begin{eqnarray}
\left\langle\left(\begin{array}{c} d_{2}\\ d_{3}\end{array}\right),\left(\begin{array}{c} \cB_E\\ \cB_I\end{array}\right)dx\right\rangle &=& \inprod{\cB_E^{*}d_{2}+\cB_I^{*}d_{3}}{dx}\nn\\
&=& \sum_{i:\bar{x}_i\neq 0}(\cB_E^{*}d_{2}+\cB_I^{*}d_{3})_{i}(dx)_{i} +\sum_{i:\bar{x}_i= 0}(\cB_E^{*}d_{2}+\cB_I^{*}d_{3})_{i}(dx)_{i}\nn\\
&=& 0.\label{eq:d2d3 perp Bdx}
\end{eqnarray}
From the nondegeneracy condition \eqref{eq:constraint nondegeneracy condition2}, there exist $dx\in\textrm{T}^{\textrm{lin}}(\bar{x})$ and $\hat{d}\in\mathbb{R}^{n}$ such that $\cB_E dx=d_{2}$ and $\cB_I dx+\hat{d}=d_{3}$.
Therefore, in combination with \eqref{eq:d3 perp dhat} and \eqref{eq:d2d3 perp Bdx}, we have
\begin{eqnarray*}
\left\langle\left(\begin{array}{c} d_{2}\\ d_{3}\end{array}\right),\left(\begin{array}{c} d_{2}\\ d_{3}\end{array}\right)\right\rangle &=& \left\langle\left(\begin{array}{c} d_{2}\\ d_{3}\end{array}\right),\left(\begin{array}{cc} \cB_E & 0\\ \cB_I & I\end{array}\right)\left(\begin{array}{c} dx\\ \hat{d}\end{array}\right)\right\rangle\\
&=&\left\langle\left(\begin{array}{c} \cB_E^{*}d_{2}+\cB_I^{*}d_{3}\\ d_{3}\end{array}\right),\left(\begin{array}{c} dx\\ \hat{d}\end{array}\right)\right\rangle\\
&=&0.
\end{eqnarray*}
Now we have proved that the elements of the generalized Jacobian $\hat{\partial}^{2}\varphi(\bar{u},\bar{v}_{E},\bar{v}_{I})$ are all positive definite.\\
``(ii)$\Rightarrow$(iii)": This is obvious.\\
``(iii)$\Rightarrow$(i)": For this result, we prove it by contradiction. We assume that \eqref{eq:constraint nondegeneracy condition2} does not hold. Then for $d=\left(\begin{array}{c} d_{1}\\ d_{2}\\ d_{3}\end{array}\right)$ with $d_{1}\in\mathbb{R}^{m}$, $d_{2}\in\mathbb{R}^{m_{E}}$, $d_{3}\in\mathbb{R}^{m_{I}}$, we can find
\begin{eqnarray*}
0\neq d_2\in[\cB_{E}\textrm{T}^{\textrm{lin}}(\bar{x})]^{\perp},
\end{eqnarray*}
and
\begin{eqnarray*}
0\neq d_3\in[\cB_{I}\textrm{T}^{\textrm{lin}}(\bar{x})]^{\perp}\cap[\textrm{lin}(T_{\mathbb{R}^{m_I}_{-}}(\bar{z}))]^{\perp}.
\end{eqnarray*}

Since $d_2\in[\cB_{E}\textrm{T}^{\textrm{lin}}(\bar{x})]^{\perp}$, then for any $dx\in \textrm{T}^{\textrm{lin}}(\bar{x})$,
\begin{eqnarray*}
\inprod{d_2}{\cB_{E}dx} = \inprod{\cB_{E}^{*}d_2}{dx} = 0,
\end{eqnarray*}
i.e.,
\begin{eqnarray}
\label{eq:B_E*d2-in-linT-orth}
\cB_{E}^{*}d_2 \in [\textrm{T}^{\textrm{lin}}(\bar{x})]^{\perp} = \Big\{\tilde{d}\in\mathbb{R}^{n}\,|\,\tilde{d}_{i}=0\textrm{ if } \bar{x}_{i}\neq 0\Big\}.
\end{eqnarray}
We may only consider the $j$th group. It is discussed in two cases. (i). $\bar{x}_{G_j}\neq 0$: if $\bar{x}_{i}\neq 0$ for some $i\in G_{j}$, then $(\cB_{E}^{*}d_2)_{i}=0$ by \eqref{eq:B_E*d2-in-linT-orth}; if $\bar{x}_{i} = 0$ for some $i\in G_{j}$, then from \eqref{eq:prox-sig*p2} we have $(\mbox{Prox}_{\sigma p_2}(\bar{x}-\sigma (\cA^{*}\bar{u}+\cB_{E}^{*}\bar{v}_{E}+\cB_{I}^{*}\bar{v}_{I})))_{i}=0$, therefore $\theta_{i}=0$ by \eqref{eq:theta}. Hence, according to \eqref{eq:partial Prox_sigma*p}, we have $(V_2^{0})_{G_jG_j}(\cB_{E}^{*}d_2)_{G_j}=0$. (ii). $\bar{x}_{G_j} = 0$: from \eqref{eq:prox-sig*p}, we can obtain $(\mbox{Prox}_{\sigma p_2}(\bar{x}-\sigma (\cA^{*}\bar{u}+\cB_{E}^{*}\bar{v}_{E}+\cB_{I}^{*}\bar{v}_{I})))_{G_j}=0$. By \eqref{eq:partial Prox_sigma*p}, we also have $(V_2^{0})_{G_jG_j}(\cB_{E}^{*}d_2)_{G_j}=0$. From the two cases, we can find $V_{2}^{0}\in\hat{\partial}\mbox{Prox}_{\sigma p}(\bar{x}-\sigma (\cA^{*}\bar{u}+\cB_{E}^{*}\bar{v}_{E}+\cB_{I}^{*}\bar{v}_{I}))$ such that $V_{2}^{0}(\cB_{E}^{*}d_2)=0$. In the same way, we can get $V_{2}^{0}(\cB_{I}^{*}d_3)=0$. Therefore, $V_{2}^{0}(\cB_{E}^{*}d_2+\cB_{I}^{*}d_3)=0$.
Meanwhile,
\begin{eqnarray}
d_3\in[\textrm{lin}(T_{\mathbb{R}^{m_I}_{-}}(\bar{z}))]^{\perp}:= \Big\{\check{d}\in\mathbb{R}^{m_I}\,|\,\check{d}_{i}\in\mathbb{R},\textrm{ if }\bar{z}_i=0;\,\,\check{d}_{i}=0,\textrm{ if }\bar{z}_i<0\Big\}.
\label{eq:orth-lineality-R-}
\end{eqnarray}
If $\bar{z}_{i}=0$, then $(\sigma \bar{v}_{I}-\bar{z})_{i}\leq 0$. In combination with \eqref{eq:Jacobian-Projector-R+} and \eqref{eq:orth-lineality-R-}, we can deduce $V_{3}^{0}d_{3}=0$.

In summery, we have found a $d\neq 0$ such that $\inprod{d}{\widehat{H}d}=0$, which contradicts the positive definiteness of $\widehat{H}$. Hence, the primal constraint nondegenaracy condition (\ref{eq:constraint nondegeneracy condition2}) holds.
\end{proof}

\subsection{Local convergence rate of the SSN}
\label{subsec:Local convergence rate of the SSN}

In this subsection, we analyze the local convergence rate of the SSN. But before that, we need to consider the strong
semismoothness of $\nabla\varphi$.  For the definitions of semismoothness and $\gamma$-order semismoothness, one
may see the following.
\begin{definition}
(semismoothness) \cite{Kummer1988Newton,QiS1993} Let $\mathcal{O} \subseteq \mathcal{R}^{n}$ be an open set, $\mathcal{K}:\mathcal{O}\subseteq \mathcal{R}^{n}\rightrightarrows \mathcal{R}^{n \times m}$ be a nonempty and compact valued, upper-semicontinuous set-valued mapping, and $F:\mathcal{O}\rightarrow \mathcal{R}^{n}$ be a locally Lipschitz continuous function. $F$ is said to be semismooth at $x \in O$ with respect to the multifunction $\mathcal{K}$ if $F$ is directionally differentiable at $x$ and for any $V \in \mathcal{K}(x+\Delta x)$ with $\Delta x\rightarrow 0$,
\begin{equation*}
  F(x+\Delta x)-F(x)-V\Delta x=o(\|\Delta x\|).
\end{equation*}
Let $\gamma$ be a positive constant. $F$ is said to be $\gamma$-order (strongly, if $\gamma=1$) semismooth at $X$ with respect to $\mathcal{K}$ if $F$ is directionally differentiable at $x$ and for any $V \in \mathcal{K}(x+\Delta x)$ with $\Delta x\rightarrow 0$,
\begin{equation*}
  F(x+\Delta x)-F(x)-V\Delta x=O(\|\Delta x\|^{1+\gamma}).
\end{equation*}
\end{definition}

In the following, we present the following result about the semismoothness of $\nabla\varphi$.
\begin{proposition}
$\nabla\varphi$ is strongly semismooth.
\end{proposition}
\begin{proof}
 Based on Proposition 7.4.7 in \cite{Facchinei2003Finite}, we have $\Pi_{\mathbb{R}^{m_I}_{+}}(\cdot)$ is strongly semismooth since it is piecewise affine. By \cite[Proposition 4.3]{ChenSS2003}, the projection operator onto the second order cone is strongly semismooth. Then the strong semismoothness of the proximal operator $\textrm{Prox}_{\sigma h}(\cdot)$ follows from \cite[Theorem 4]{MengSZ2005}.
  In \cite[Theorem 3.1]{ZhangZST}, Zhang et al. have proved that $\mbox{Prox}_{\sigma p}(\cdot)$ is strongly semismooth. Finally, by Theorem 7.5.17 in \cite{Facchinei2003Finite}, we can easily prove the conclusion and we omit the details.
\end{proof}

Since the strong semismoothness of $\nabla\varphi$ has been proved, we are in the position to state the local convergence rate of the SSN method.
\begin{theorem}\label{theorem_SsN_convergence}
Let $\{(u^{j},v_E^{j},v_I^{j},w^{j},s^{j})\}$ be the infinite sequence generated by Algorithm SSN. Then $\{(u^{j},v_E^{j},v_I^{j},w^{j},s^{j})\}$ converges to the unique optimal solution $(\tilde{u},\tilde{v}_E,\tilde{v}_I,\tilde{w},\tilde{s})$ to problem \eqref{eq:subproblem of alm} and
\begin{eqnarray*}
\|(u^{j+1},v_E^{j+1},v_I^{j+1},w^{j+1},s^{j+1})-(\tilde{u},\tilde{v}_E,\tilde{v}_I,\tilde{w},\tilde{s})\|=O(\|(u^{j},v_E^{j},v_I^{j},w^{j},s^{j})-(\tilde{u},\tilde{v}_E,\tilde{v}_I,\tilde{w},\tilde{s})\|^{1+\tau}.
\end{eqnarray*}
\end{theorem}

\subsection{Convergence analysis of the ALM}
\label{subsec:Convergence analysis of the ALM}

In this subsection, we adapt the results developed in \cite{Rockafellar76a,Rockafellar76b,Luque} to establish the convergence theory of the ALM for problem (D).

Notably, the inner subproblem \eqref{eq:subproblem of alm} has no closed-form solution, so we consider how to solve it approximately with the following stopping criteria introduced in \cite{Rockafellar76a,Rockafellar76b}.
\begin{equation*}\label{Def_stop_cond_1}
\begin{split}
\text{(A)}&\hspace*{2mm} \varphi_{k}(u^{k+1},v_E^{k+1},v_I^{k+1})-\inf_{u,v_E,v_I}\varphi_{k}(u,v_E,v_I)\leq \epsilon_{k}^{2}/(2\sigma_{k}),\,\epsilon_{k}\geq 0,\,\sum_{k=1}^{+\infty}\epsilon_{k}<+\infty,\\
\text{(B)}&\hspace*{2mm} \varphi_{k}(u^{k+1},v_E^{k+1},v_I^{k+1})-\inf_{u,v_E,v_I}\varphi_{k}(u,v_E,v_I)\leq \delta_{k}^{2}/(2\sigma_{k})(\norm{x^{k+1}-x^{k}}^{2}+\norm{y^{k+1}-y^{k}}^{2}+\\
&\hspace*{2mm} \norm{z^{k+1}-z^{k}}^{2}),\,\delta_{k}\geq 0,\,\sum_{k=1}^{+\infty}\delta_{k}<+\infty,\\
\text{(B$'$)}&\hspace*{2mm} \norm{\nabla\varphi_{k}(u^{k+1},v_E^{k+1},v_I^{k+1})}\leq \delta_{k}'/(2\sigma_{k})(\norm{x^{k+1}-x^{k}}^{2}+\norm{y^{k+1}-y^{k}}^{2}+\norm{z^{k+1}-z^{k}}^{2})^{\frac{1}{2}},\\
&\hspace*{2mm} \, 0\leq\delta_{k}'\rightarrow 0.
\end{split}
\end{equation*}

Then the global convergence of Algorithm SSN-ALM follows from \cite[Theorem 1]{Rockafellar76a} and
\cite[Theorem 4]{Rockafellar76b} without much difficulty.
\begin{theorem}
Suppose $\inf(D)<+\infty$, and let Algorithm SSN-ALM be executed with the stopping criterion (A). If Assumption \ref{assumption:Slater-D} holds, then the sequence $(x^{k},y^{k},z^{k})$  generated by the algorithm is bounded and $(x^{k},y^{k},z^{k})$  converges to $(\bar{x},\bar{y},\bar{z})$, where $(\bar{x},\bar{y},\bar{z})$ is some optimal solution to (P), and $(u^{k},v_{E}^{k},v_{I}^{k},w^{k},s^{k})$ is asymptotically minimizing
for (D) with $\inf(P) = \max(D)$.

If $(x^{k},y^{k},z^{k})$ is bounded and (P) satisfies Assumption \ref{assumption:Slater-P}, then the sequence $(u^{k},v_{E}^{k},v_{I}^{k},w^{k},s^{k})$ is also
bounded, and all of its accumulation points of the sequence $(u^{k},v_{E}^{k},v_{I}^{k},w^{k},s^{k})$ are optimal solutions
to (D).
\end{theorem}

Next, we state the local convergence rate of Algorithm SSN-ALM.
\begin{theorem}
Suppose $\inf(D)<+\infty$, and let Algorithm SSN-ALM be executed with the stopping criteria (A) and (B). Suppose Assumptions \ref{assumption:Slater-P}, \ref{assumption:Slater-D} and \ref{assumption:y != 0} hold. Suppose that $\cT_{f}$ satisfies the error bound condition \eqref{eq:error bound} for the origin with modulus $a_{f}$. Let $\{(x^{k},y^{k},z^{k},u^{k},v_{E}^{k},v_{I}^{k},w^{k},s^{k})\}$ be any infinite sequence generated by Algorithm SSN-ALM with the stopping criteria (A) and (B$'$). Then, the sequence $\{(x^{k},y^{k},z^{k})\}$ converges to $(\bar{x},\bar{y},\bar{z})\in\Omega$, where $\Omega$ is the solution set to problem (P), and for all $k$ sufficiently large,
\begin{eqnarray*}
\textrm{dist}((x^{k+1},y^{k+1},z^{k+1}),\Omega) \leq \theta_{k}\textrm{dist}((x^{k},y^{k},z^{k}),\Omega),
\end{eqnarray*}
where $\theta_{k}=(a_{f}(a_{f}^{2}+\sigma_{k}^{2})^{-1/2}+2\delta_{k})(1-\delta_{k})^{-1}\rightarrow\theta_{\infty}=a_{f}(a_{f}^{2}+\sigma_{\infty}^{2})^{-1/2}<1$, as $k\rightarrow+\infty$. Moreover, if the nondegeneracy condition \eqref{eq:constraint nondegeneracy condition2} holds, the sequence $\{(u^{k},v_{E}^{k},v_{I}^{k},w^{k},s^{k})\}$ converges to the unique optimal solution to problem (D).

Moreover, if $\cT_{l}$ satisfies the error bound condition \eqref{eq:error bound} for the origin with modulus $a_{l}$ and the stopping criterion (B$'$) is also used, then for all k sufficiently large,
\begin{eqnarray*}
\norm{(u^{k+1},v_{E}^{k+1},v_{I}^{k+1},w^{k+1},s^{k+1})-(\bar{u},\bar{v}_{E},\bar{v}_{I},\bar{w},\bar{s})} \leq \theta'_{k}\norm{(x^{k+1},y^{k+1},z^{k+1})-(x^{k},y^{k},z^{k})},
\end{eqnarray*}
where $\theta'_{k}=a_{l}(1+\delta'_{k})/\sigma_{k}$ with the limit $\lim_{k\rightarrow\infty}\theta'_{k}=a_{l}/\sigma_{\infty}$.
\end{theorem}

\section{Numerical issues for solving the subproblem \eqref{eq:subproblem of alm}}
\label{sec:Numerical issues for solving the subproblem}

The key part of Algorithm SSN-ALM is how to efficiently solve the subproblem \eqref{eq:subproblem of alm}. The most important thing to solve this subproblem is how to efficiently solve the linear system to obtain the Newton direction. Denote $\widetilde{H}:=\sigma^{-1}\widehat{H}$, then the linear system has the following form
\begin{eqnarray}
\widetilde{H}d = \left[\left(\begin{array}{ccc} V_1^{0}& & \\ &0& \\ & & V_3^{0}\end{array}\right)
+ N V_2^{0} N^{T}\right]d = -\sigma^{-1}\nabla\varphi(u,v_{E},v_{I}),\label{eq:detailed Newton system}
\end{eqnarray}
where $N\in \mathbb{R}^{\widehat{m}\times n} (\widehat{m}:=m+m_{E}+m_{I})$ denotes the matrix representation of the linear operator $\left(\begin{array}{c}\cA\\\cB_E\\ \cB_I\end{array}\right)$, $d=\left(\begin{array}{c} d_{1}\\ d_{2}\\ d_{3}\end{array}\right)\in\mathbb{R}^{\widehat{m}}$, $V_{1}^{0}=I_{m} - W_{1}^{0}$, here

\begin{eqnarray*}
W_{1}^{0} &=& \left\{\begin{array}{ll}\frac{\sigma}{\norm{y+\sigma u}}(I_{m}-\frac{(y+\sigma u)(y+\sigma u)^{T}}{\norm{y+\sigma u}^{2}}),&\mbox{if}\ \norm{y+\sigma u}>\sigma,\\
I_{m}, &\mbox{if}\ \norm{y+\sigma u}\leq\sigma,\end{array}\right.
\label{eq:W0}
\end{eqnarray*}
$V_{2}^{0}$ and $V_{3}^{0}$ are defined the same as \eqref{eq:V_2^0} and \eqref{eq:V_3^0}.
%


In \cite[Section 4.3]{ZhangZST}, Zhang et al. once described the special structure similar to $N V_2^{0} N^{T}$. For completeness, we also state the details here. Note that
\begin{eqnarray*}
\textrm{supp}(\diag(\cP_{j}^{*}\cP_{j}\Theta))=\Xi_{j},
\end{eqnarray*}
where
$\Xi_{j}$ is the index set defined in \eqref{eq:Xi_j} that corresponds to the non-zero elements of $v=\textrm{Prox}_{\sigma p_2}(x-\sigma (\cA^{*}u+\cB_{E}^{*}v_{E}+\cB_{I}^{*}v_{I}))$ in the $j$th group. The diagonal matrix $\cP^{*}_{j}\cP_{j}\Theta$ is expected to contain
only a few ones on the diagonal so that the computational cost of $N\cP^{*}_{j}\cP_{j}\Theta N^{T}$
can be greatly reduced. Denote $\Xi_{>}:=\{j\,|\,\norm{\cP_{j}v}>\sigma\lambda_{1,j},\,j=1,\ldots,J\}$.
Let $N_{j}\in\mathbb{R}^{\widehat{m}\times |\Xi_{j}|}$ be the sub-matrix of $N$ with those columns in $\Xi_{j}$ and $s_{j}:=(\cP^{*}_{j}\cP_{j}v)_{\Xi_{j}}\in\mathbb{R}^{|\Xi_{j}|}$ be the sub-vector of $\cP^{*}_{j}v_{j}$ restricted to $\Xi_{j}$. Then we have
\begin{eqnarray*}
NV_{2}^{0}N^{T} &=& N(I_{n}-\cP^{*}\Sigma \cP)\Theta N^{T} \\
&=& \sum_{j\in\Xi_{>}}\left(1-\frac{\sigma\lambda_{1,j}}{\norm{\cP_{j}v}}\right)N\cP^{*}_{j}\cP_{j}\Theta N^{T} + \frac{\sigma\lambda_{1,j}}{\norm{\cP_{j}v}^{3}}N(\cP^{*}_{j}\cP_{j}v)(\cP^{*}_{j}\cP_{j}v)^{T}N^{T}\\
&=& \sum_{j\in\Xi_{>}}\left(1-\frac{\sigma\lambda_{1,j}}{\norm{\cP_{j}v}}\right)N_{j}N_{j}^{T} + \frac{\sigma\lambda_{1,j}}{\norm{\cP_{j}v}^{3}}(N_{j}s_{j})(N_{j}s_{j})^{T}.
\end{eqnarray*}

If we let $r:=\sum_{j\in\Xi_{>}}|\Xi_{j}|$, $r_{2}:=|\Xi_{>}|$, $D=[B\ \  C]\in\mathbb{R}^{\widehat{m}\times (r+r_{2})}$ with $B_{j}:=\sqrt{(1-\frac{\sigma\lambda_{1,j}}{\norm{\cP_{j}v}})}N_{j}\in\mathbb{R}^{\widehat{m}\times |\Xi_{j}|}$, $B:=[B_{j}]_{j\in\Xi_{>}}\in\mathbb{R}^{\widehat{m}\times r}$, $c_{j}:=\sqrt{\frac{\sigma\lambda_{1,j}}{\norm{\cP_{j}v}^{3}}}(N_{j}s_{j})\in\mathbb{R}^{\widehat{m}}$ and $C:=[c_{j}]_{j\in\Xi_{>}}\in\mathbb{R}^{\widehat{m}\times r_{2}}$, then
\begin{eqnarray}
NV_{2}^{0}N^{T} &=& DD^{T}.\label{eq:NV2V}
\end{eqnarray}
For $V_{1}^{0}$, we only analyze the case of $\norm{y+\sigma u}>1$, which is relatively most complicated. That is,
\begin{eqnarray}
V_{1}^{0} = \left(1-\frac{\sigma}{\norm{y+\sigma u}}\right)I_{m}-\sigma\frac{(y+\sigma u)(y+\sigma u)^{T}}{\norm{y+\sigma u}^{3}}. \label{eq:V_1^0}
\end{eqnarray}
Combining \eqref{eq:NV2V} and \eqref{eq:V_1^0}, the coefficient matrix of \eqref{eq:detailed Newton system} has the form
\begin{eqnarray*}
\widetilde{H} = \left(\begin{array}{ccc} \left(1-\frac{\sigma}{\norm{y+\sigma u}}\right)I_{m}-\frac{\sigma(y+\sigma u)(y+\sigma u)^{T}}{\norm{y+\sigma u}^{3}}& & \\ &0& \\ & & V_3^{0}\end{array}\right)
+ DD^{T}.\label{eq:coefficient matrix of the Newton system}
\end{eqnarray*}
Due to the sparse structure, the number $r+r_{2}$ may be much smaller than $\widehat{m}$. Therefore, by exploiting the second order sparsity the total computational costs of using the sparse Cholesky factorization to solve the linear system \eqref{eq:detailed Newton system} are significantly reduced from $O(\widehat{m}^{3})$ to $O(\widehat{m}(r+r_{2})^{2})$. But if the dimension of the problem is very large, we tend to apply the PCG method instead of the sparse Cholesky factorization method because it is more efficient in practical computation. When using the PCG method, we adopt a diagonal preconditioner.

\section{Numerical experiments}
\label{sec:Numerical experiments}

In this section, we compare the performances of our proposed algorithm with the semi-proximal ADMM on the cssLasso problems. We implement the algorithm in {\sc Matlab} R2019a. All runs are performed on a NoteBook (i710710u 4.7G with 16 GB RAM).

In our experiments, we measure the quality of the computed solution by the relative primal infeasibility $R_{P}$ and dual infeasibility $R_{D}$, and the relative complementarity condition $R_{C}$ as follows
\begin{eqnarray*}
R_{P} &=& \frac{\norm{\cA x-y-b}+\norm{\cB_{E}x-c_{E}}+\norm{\cB_{I}x-c_{I}+z}}{1+\norm{b}+\norm{c_{E}}+\norm{c_{I}}},\\
R_{D} &=& \frac{\norm{\mathcal{A}^{*}u + \cB^{*}_{E} v_{E} + \cB^{*}_{I} v_{I} + s}+\norm{w-u}}{1+\norm{u}+\norm{v_{E}}+\norm{v_{I}}+\norm{s}+\norm{w}},\\
R_{C} &=& \frac{\norm{w-\textrm{Prox}_{h^*}(w+y)}+\norm{s-\textrm{Prox}_{p^*}(s+x)}+\norm{v_{I}-\Pi_{\mathbb{R}^{m_I}_{-}}(\cB_{I}x-c_{I}+v_{I})}}{1+\norm{w}+\norm{s}+\norm{v_{I}}}.
\end{eqnarray*}
We stop the algorithm when
\begin{eqnarray*}
\eta_{kkt}:=\max\{R_{P},R_{D},R_{C}\} < \texttt{Tol},
\end{eqnarray*}
with $\texttt{Tol}=10^{-6}$ as the default. We also stop the algorithm if it reaches the maximum iteration number 200 for the SSN-ALM and 10000 for the semi-proximal ADMM. Meanwhile, we set the maximum running time to be $4$ hours. In addition, we also use the relative gap $R_{G}$ as a measurement, which is defined as
\begin{eqnarray*}
R_{G}:=\frac{|\pobj-\dobj|}{1+|\pobj|+|\dobj|},
\end{eqnarray*}
where $\pobj$ and $\dobj$ denote the primal and dual objective values, respectively.
The weights $\omega_{j}=\sqrt{|G_{j}|},\,\forall j=1,2,\ldots,J$. We test the problems with two different sets of regularization parameters:
\begin{eqnarray*}
(S_{1}) & \lambda_{1}=0.5\gamma\norm{\cA^{*}b}_{\infty},\,\lambda_{2}=0.5\gamma\norm{\cA^{*}b}_{\infty};\\
(S_{2}) & \lambda_{1}=0.8\gamma\norm{\cA^{*}b}_{\infty},\,\lambda_{2}=0.2\gamma\norm{\cA^{*}b}_{\infty},
\end{eqnarray*}
where the parameter $\gamma$ is chosen to produce a reasonable number of nonzero elements in the resulting solution $x$.
Let $\hat{x}$ be the vector obtained by sorting $x$ such that $|\hat{x}_{1}|\geq|\hat{x}_{2}|\geq\cdots\geq|\hat{x}_{n}|$. In our numerical experiments, we define the number of nonzero elements as the minimal $k$ satisfying
\begin{eqnarray*}
\sum_{i=1}^{k}|\hat{x}_{i}|\geq 0.9999\|x\|_{1}.
\end{eqnarray*}

\subsection{The semi-proximal ADMM for the dual problem ($D$)}

For the clarity, we present the details about the semi-proximal ADMM for the dual problem ($D$). For a given $\sigma>0$, the semi-proximal augmented Lagrangian function associated with the dual problem $(D)$ is defined by
\begin{eqnarray*}
\widehat{L}_{\sigma}(u,v_{E},v_{I},\hat{v}_{I},w,s;x,y,z)&:=&h^{*}(w)+p^{*}(s)+\delta_{\mathbb{R}^{m_I}_{-}}(\hat{v}_{I})+\inprod{b}{u}+\inprod{c_{E}}{v_{E}}+\inprod{c_{I}}{v_{I}}\\
&&+\frac{\sigma}{2}\|\mathcal{A}^{*}u+\cB^{*}_{E}v_{E}+\cB^{*}_{I}v_{I}+s-\sigma^{-1}x\|^{2}-\frac{1}{2\sigma}\|x\|^{2}\\
&&+\frac{\sigma}{2}\|w-u-\sigma^{-1}y\|^{2}-\frac{1}{2\sigma}\|y\|^{2}+\frac{\sigma}{2}\|v_{I}-\hat{v}_{I}-\sigma^{-1}z\|^{2}\\
&&-\frac{1}{2\sigma}\|z\|^{2} + \frac{\tilde{\tau}}{2\sigma}\norm{v_{E}-v_{E}^{k}}^{2}.
\end{eqnarray*}

If we minimize the above function with respect to the variables $\hat{v}_{I},w$ and $s$, they have the closed-form solution. As for the minimization with respect to $u, v_{E}$ and $v_{I}$ at the $k$th iteration, we need to solve the following linear system of equations
\begin{eqnarray*}
M
\left(\begin{array}{l}
u\\v_{E}\\v_{I}
\end{array}\right)
=
\textrm{rhs}^{k},
\label{eq:linear system for ADMM}
\end{eqnarray*}
where
\begin{eqnarray*}
M=\left[
\left(\begin{array}{l}
\cA\\\cB_{E}\\\cB_{I}
\end{array}\right)
\left(\begin{array}{lll}
\cA^{*}\quad \cB_{E}^{*}\quad \cB_{I}^{*}
\end{array}\right)
+
\left(\begin{array}{lll}
\cI\quad\ \ \ 0\quad\ \ \ 0\\
0\quad \tilde{\tau}\sigma^{-2}\cI\quad 0\\
0\quad\ \ \ 0\quad\ \ \ \cI\\
\end{array}\right)
\right],
\end{eqnarray*}
and
\begin{eqnarray*}
\textrm{rhs}^{k}=
\left(\begin{array}{l}
-\cA(s^{k}-\sigma^{-1}x^{k})+(w^{k}-\sigma^{-1}y^{k})-\sigma^{-1}b\\
-\cB_{E}(s^{k}-\sigma^{-1}x^{k})+\tilde{\tau}\sigma^{-2}v_{E}^{k}-\sigma^{-1}c_{E}\\
-\cB_{I}(s^{k}-\sigma^{-1}x^{k})+(\hat{v}_{I}^{k}+\sigma^{-1}z^{k})-\sigma^{-1}c_{I}
\end{array}\right).
\end{eqnarray*}

Based on the above analysis, the semi-proximal ADMM can be stated as follows

\bigskip
\ni\fbox{\parbox{\textwidth}{\noindent{\bf Algorithm semi-proximal ADMM:} Let $\sigma>0$ be a given parameter, $\tau\in\left(0,\frac{1+\sqrt{5}}{2}\right)$. Select an initial point $(u^{0},v_{E}^{0},v_{I}^{0},\hat{v}_{I}^{0},w^{0},s^{0},x^{0},y^{0},z^{0})\in \mathbb{R}^{m}\times\mathbb{R}^{m_E}\times\mathbb{R}^{m_I}\times\mathbb{R}^{m_I}\times\mathbb{R}^{m}\times\mathbb{R}^{n}\times \mathbb{R}^{n}\times\mathbb{R}^{m}\times\mathbb{R}^{m_I}$. For $k=0,1,\ldots$, iterate the following steps.
\begin{description}
\item [Step 1.] Compute
\begin{eqnarray*}
\left(\begin{array}{l}
u^{k+1}\\v_{E}^{k+1}\\v_{I}^{k+1}
\end{array}\right)
&=&
M^{-1}\textrm{rhs}^{k},\\
\left(\begin{array}{l}
\hat{v}_{I}^{k+1}\\w^{k+1}\\s^{k+1}
\end{array}\right)
&=&
\left(\begin{array}{c}
\Pi_{\mathbb{R}^{m_I}_{-}}(v_{I}^{k+1}-\sigma^{-1}z^{k})\\
\textrm{Prox}_{h^{*}/\sigma}(\sigma^{-1}y^{k}+u^{k+1})\\
\textrm{Prox}_{p^{*}/\sigma}(\sigma^{-1}x^{k}-\cA^{*}u^{k+1}-\cB_{E}^{*}v_{E}^{k+1}-\cB_{I}^{*}v_{I}^{k+1})
\end{array}\right).
\label{eq:subproblem for ADMM}
\end{eqnarray*}

\item [Step 2.] Compute
\begin{eqnarray*}
\left\{\begin{array}{l}
x^{k+1}=x^{k}-\tau\sigma(\mathcal{A}^{*}u^{k+1}+\cB_{E}^{*}v_{E}^{k+1}+\cB_{I}^{*}v_{I}^{k+1}+s^{k+1}),\\
y^{k+1}=y^{k}-\tau\sigma(w^{k+1}-u^{k+1}),\\
z^{k+1}=z^{k}-\tau\sigma(v_{I}^{k+1}-\hat{v}_{I}^{k+1}).
\end{array}\right.
\end{eqnarray*}
\end{description}}}

\bigskip

\subsection{The description of the computed problems}
\label{subsec:The description of the computed problems}

In the numerical experiments, we focus on three classes of special problems.

\paragraph{\uppercase\expandafter{\romannumeral1}: the general constrained problem.}

The general constrained problem has the following form
\begin{eqnarray}
\min_{x\in\mathbb{R}^{n}}\Big\{\norm{A x-b}+\lambda_1\sum_{j=1}^{J}\omega_{j}\norm{x_{G_j}}+\lambda_2\norm{x}_{1}\,\Big|\,B_{E}x=0,\, \,B_{I}x\geq 0\Big\},\label{eq:general constrained problem}
\end{eqnarray}
where $A\in\mathbb{R}^{m\times n}$, $B_{E}\in\mathbb{R}^{m_{E}\times n}$ with the $i$th row $(B_{E})_{i,:}:=[0_{k_i}^{T}\; 0_{k_i}^{T}\; \cdots\; 1_{k_i}^{T}\; 0_{k_i}^{T} \cdots 1_{k_i}^{T}]$, $l_{b}=0_{n}$, and $l_{u}=1_{n}$, and $B_{I}\in\mathbb{R}^{m_{I}\times n}$ is generated in the same way as $B_{E}$. Note that $0_{k_i}\in\mathbb{R}^{k_i}$ and $0_{n}\in\mathbb{R}^{n}$ are vectors of all zeros, $1_{k_i}\in\mathbb{R}^{k_i}$ and $1_{n}\in\mathbb{R}^{n}$ are vectors of all ones. As for the details for the equality constraint and inequality constraint, one may see \cite{GainesZ2018}.

%

\paragraph{\uppercase\expandafter{\romannumeral2}: the reparameterization problem.}

The equality constrained problem has the following form
\begin{eqnarray}
\min_{x\in\mathbb{R}^{n}}\Big\{\norm{A x-b}+\lambda_1\sum_{j=1}^{J}\omega_{j}\norm{x_{G_j}}+\lambda_2\norm{x}_{1}\,\Big|\,B_{E}x=0\Big\},\label{eq:reparameterization}
\end{eqnarray}
where $A\in\mathbb{R}^{m\times n}$ and $B_{E}\in\mathbb{R}^{m_{E}\times n}$. Note that $B_{E}$ is generated in the same way as that in the first class of problem.

\paragraph{\uppercase\expandafter{\romannumeral3}: the sum-to-zero constraint problem.}

The problem is as follows
\begin{eqnarray}
\min_{x\in\mathbb{R}^{n}}\Big\{\norm{A x-b}+\lambda_1\sum_{j=1}^{J}\omega_{j}\norm{x_{G_j}}+\lambda_2\norm{x}_{1}\,\Big|\,\sum_{i}x_{i}=0\Big\},\label{eq:sum-to-zero constraints}
\end{eqnarray}
where $A\in\mathbb{R}^{m\times n}$. This type of constraint on the Lasso has been considered in \cite{AltenbuchingerRZSDWHGHOS,LinSFL,ShiZL} for the analysis of compositional data.

%

%

\subsection{Numerical results on the synthetic data problems}
\label{subsec:Numerical results on the synthetic data problems}

In this subsection, we compare the the numerical performances between the SSN-ALM and the semi-proximal ADMM on the synthetic data problems. For the synthetic data problems, we only focus on the problem \eqref{eq:general constrained problem}. In our comparison, we report the problem (pbname), the number of samples ($m$), features ($n$) and nonzero elements (nnz), $\lambda_{1}$, $\lambda_{2}$, the relative KKT residual ($\eta_{kkt}$), the primal objective value (pobj), the iteration number (iter) (for the SSN-ALM, it also includes the total number of Newton iterations in the bracket) and the running time (time) in the format of ``hours:minutes:seconds''. For simplicity, we use``$s\ \mbox{sign}(t)|t|$'' to denote a number of the form ``$s\times 10^{t}$'', e.g., 1.0-3 denotes $1.0\times 10^{-3}$. We present the numerical results in Tables \ref{table-random1} and \ref{table-random2}. The numerical results show that the SSN-ALM can solve all the synthetic problems efficiently and supply highly accurate solutions. In contrast, the semi-proximal ADMM can only solve a few problems in much more time, and for most problems it cannot supply satisfying solutions within the maximum time.

\begin{table}[!htbp]
\scriptsize
\begin{center}
\setlength{\belowcaptionskip}{10pt}
\parbox{.85\textwidth}{\caption{The performances of the SSN-ALM and semi-proximal ADMM on the synthetic datasets for the problem \eqref{eq:general constrained problem} with the parameter setting ($S_{1}$). In this table, ``$a$''=SSN-ALM, ``$b$"=semi-proximal ADMM.}\label{table-random1}}
\begin{tabular}{|p{2.5cm}<{\centering}|p{0.8cm}<{\centering}|p{0.8cm}<{\centering}|p{0.4cm}<{\centering}|p{1.6cm}<{\centering}|p{2.8cm}<{\centering}|p{1.9cm}<{\centering}|p{1.7cm}<{\centering}|}
\hline
\multirow{1}{*}{pbname}&\multirow{3}{*}{$\lambda_{1}$}&\multirow{3}{*}{$\lambda_{2}$}&\multirow{3}{*}{nnz}&\multicolumn{1}{c|}{$\eta_{kkt}$}&\multicolumn{1}{c|}{$\pobj$}&\multicolumn{1}{c|}{iter}&\multicolumn{1}{c|}{time}\\
\cline{5-8} $(m,n,m_{E},m_{I})$ & & & & \multirow{2}{*}{$a\ |\ b$} & \multirow{2}{*}{$a\ |\ b$} & \multirow{2}{*}{$a\ |\ b$} &\multirow{2}{*}{$a\ |\ b$}\\[2pt]
$J$ & & & & & & & \\
\hline
\mc{1}{|c|}{rand1}
&\multirow{1}{*}{6.638-3} &\multirow{1}{*}{6.638-3} &\multirow{1}{*}{249}&\multirow{1}{*}{ 9.8-7 $|$ 9.9-7} &\multirow{1}{*}{ 3.2601+2 $|$  3.2600+2} &\multirow{1}{*}{18(95) $|$ 3287}  &\multirow{1}{*}{03 $|$ 54:47}
\\[2pt]
\mc{1}{|c|}{(100,10000,24,24);} &\multirow{1}{*}{6.638-4} &\multirow{1}{*}{6.638-4} &\multirow{1}{*}{249} &\multirow{1}{*}{ 6.1-7 $|$  9.9-7} &\multirow{1}{*}{ 3.2606+1 $|$  3.2600+1} &\multirow{1}{*}{21(104) $|$ 8359} &\multirow{1}{*}{04 $|$ 31:21}\\[2pt]
\mc{1}{|c|}{1000} &\multirow{1}{*}{6.638-5} &\multirow{1}{*}{6.638-5} &\multirow{1}{*}{252}&\multirow{1}{*}{ 9.3-7 $|$  1.4-4} &\multirow{1}{*}{3.2777+0 $|$  3.2827+0} &\multirow{1}{*}{33(180) $|$ 10000}  &\multirow{1}{*}{06 $|$ 43:13}\\[2pt]
\hline
\mc{1}{|c|}{rand2}
&\multirow{1}{*}{9.458-2} &\multirow{1}{*}{9.458-2} &\multirow{1}{*}{184}&\multirow{1}{*}{ 7.6-7 $|$ 2.5-1} &\multirow{1}{*}{ 1.3391+3 $|$  1.4223+3} &\multirow{1}{*}{18(116) $|$ 402}  &\multirow{1}{*}{2:11 $|$ 4:00:09}
\\[2pt]
\mc{1}{|c|}{(100,1000000,24,24);} &\multirow{1}{*}{9.458-3} &\multirow{1}{*}{9.458-3} &\multirow{1}{*}{184} &\multirow{1}{*}{ 7.3-7 $|$  3.0+0} &\multirow{1}{*}{ 1.3391+2 $|$  1.4056+2} &\multirow{1}{*}{24(133) $|$ 404} &\multirow{1}{*}{2:40 $|$ 4:00:01}\\[2pt]
\mc{1}{|c|}{100000} &\multirow{1}{*}{9.458-4} &\multirow{1}{*}{9.458-4} &\multirow{1}{*}{184}&\multirow{1}{*}{ 7.6-7 $|$  5.0+0} &\multirow{1}{*}{1.3391+1 $|$  1.4711+1} &\multirow{1}{*}{18(116) $|$ 406}  &\multirow{1}{*}{2:16 $|$ 4:00:02}\\[2pt]
\hline
\mc{1}{|c|}{rand3}
&\multirow{1}{*}{1.706-1} &\multirow{1}{*}{1.706-1} &\multirow{1}{*}{165}&\multirow{1}{*}{ 6.5-7 $|$ 1.0-1} &\multirow{1}{*}{ 3.8494+3 $|$  4.8077+3} &\multirow{1}{*}{19(115) $|$ 140}  &\multirow{1}{*}{13 $|$ 4:01:41}
\\[2pt]
\mc{1}{|c|}{(100,3000000,24,24);} &\multirow{1}{*}{1.706-2} &\multirow{1}{*}{1.706-2} &\multirow{1}{*}{166} &\multirow{1}{*}{ 9.1-7 $|$  1.5+0} &\multirow{1}{*}{ 3.8493+2 $|$  4.6124+2} &\multirow{1}{*}{25(148) $|$ 140} &\multirow{1}{*}{20:38 $|$ 4:00:10}\\[2pt]
\mc{1}{|c|}{300000} &\multirow{1}{*}{1.706-3} &\multirow{1}{*}{1.706-3} &\multirow{1}{*}{166}&\multirow{1}{*}{ 9.9-7 $|$  4.8+0} &\multirow{1}{*}{3.8490+1 $|$  4.7712+1} &\multirow{1}{*}{69(373) $|$ 134}  &\multirow{1}{*}{35:03 $|$ 4:01:45}\\[2pt]
\hline
\end{tabular}
\end{center}
\end{table}

\begin{table}[!htbp]
\scriptsize
\begin{center}
\setlength{\belowcaptionskip}{10pt}
\parbox{.85\textwidth}{\caption{The performances of the SSN-ALM and semi-proximal ADMM on the synthetic datasets for the problem \eqref{eq:general constrained problem} with the parameter setting ($S_{2}$). In this table, ``$a$''=SSN-ALM, ``$b$"=semi-proximal ADMM.}\label{table-random2}}
\begin{tabular}{|p{2.5cm}<{\centering}|p{0.8cm}<{\centering}|p{0.8cm}<{\centering}|p{0.4cm}<{\centering}|p{1.6cm}<{\centering}|p{2.8cm}<{\centering}|p{1.9cm}<{\centering}|p{1.7cm}<{\centering}|}
\hline
\multirow{1}{*}{pbname}&\multirow{3}{*}{$\lambda_{1}$}&\multirow{3}{*}{$\lambda_{2}$}&\multirow{3}{*}{nnz}&\multicolumn{1}{c|}{$\eta_{kkt}$}&\multicolumn{1}{c|}{$\pobj$}&\multicolumn{1}{c|}{iter}&\multicolumn{1}{c|}{time}\\
\cline{5-8} $(m,n,m_{E},m_{I})$ & & & & \multirow{2}{*}{$a\ |\ b$} & \multirow{2}{*}{$a\ |\ b$} & \multirow{2}{*}{$a\ |\ b$} &\multirow{2}{*}{$a\ |\ b$}\\[2pt]
$J$ & & & & & & & \\
\hline
\mc{1}{|c|}{rand1}
&\multirow{1}{*}{1.062-2} &\multirow{1}{*}{2.655-3} &\multirow{1}{*}{499}&\multirow{1}{*}{ 5.7-7 $|$ 9.9-7} &\multirow{1}{*}{ 3.7078+2 $|$  3.7073+2} &\multirow{1}{*}{19(95) $|$ 3762}  &\multirow{1}{*}{03 $|$ 13:39}
\\[2pt]
\mc{1}{|c|}{(100,10000,24,24);} &\multirow{1}{*}{1.062-3} &\multirow{1}{*}{2.655-4} &\multirow{1}{*}{499} &\multirow{1}{*}{ 8.9-7 $|$  9.9-7} &\multirow{1}{*}{ 3.7132+1 $|$  3.2600+1} &\multirow{1}{*}{22(107) $|$ 8861} &\multirow{1}{*}{04 $|$ 33:47}\\[2pt]
\mc{1}{|c|}{1000} &\multirow{1}{*}{1.062-4} &\multirow{1}{*}{2.655-5} &\multirow{1}{*}{502}&\multirow{1}{*}{ 9.4-7 $|$  5.3-5} &\multirow{1}{*}{3.7696+0 $|$  3.7283+0} &\multirow{1}{*}{31(160) $|$ 10000}  &\multirow{1}{*}{06 $|$ 43:09}\\[2pt]
\hline
\mc{1}{|c|}{rand2}
&\multirow{1}{*}{1.513-1} &\multirow{1}{*}{3.783-2} &\multirow{1}{*}{489}&\multirow{1}{*}{ 4.9-7 $|$ 2.1-1} &\multirow{1}{*}{ 6.9436+2 $|$  7.3304+2} &\multirow{1}{*}{17(101) $|$ 364}  &\multirow{1}{*}{2:01 $|$ 4:00:22}
\\[2pt]
\mc{1}{|c|}{(100,1000000,24,24);} &\multirow{1}{*}{1.513-2} &\multirow{1}{*}{3.783-3} &\multirow{1}{*}{488} &\multirow{1}{*}{ 7.9-7 $|$  1.4+0} &\multirow{1}{*}{ 6.9434+1 $|$  7.2685+1} &\multirow{1}{*}{23(118) $|$ 383} &\multirow{1}{*}{2:48 $|$ 4:00:10}\\[2pt]
\mc{1}{|c|}{100000} &\multirow{1}{*}{1.513-3} &\multirow{1}{*}{3.783-4} &\multirow{1}{*}{492}&\multirow{1}{*}{ 9.1-7 $|$  4.2+0} &\multirow{1}{*}{6.9425+0 $|$  7.5660+0} &\multirow{1}{*}{55(272) $|$ 359}  &\multirow{1}{*}{5:38 $|$ 4:00:23}\\[2pt]
\hline
\mc{1}{|c|}{rand3}
&\multirow{1}{*}{2.730-3} &\multirow{1}{*}{6.826-4} &\multirow{1}{*}{513}&\multirow{1}{*}{ 3.9-7 $|$ 7.3-2} &\multirow{1}{*}{ 2.0434+3 $|$  2.4657+3} &\multirow{1}{*}{18(98) $|$ 119}  &\multirow{1}{*}{14:09 $|$ 4:02:50}
\\[2pt]
\mc{1}{|c|}{(100,3000000,24,24);} &\multirow{1}{*}{2.730-4} &\multirow{1}{*}{6.826-5} &\multirow{1}{*}{513} &\multirow{1}{*}{ 9.0-7 $|$  1.5+0} &\multirow{1}{*}{ 2.0435+2 $|$  2.3764+2} &\multirow{1}{*}{22(122) $|$ 115} &\multirow{1}{*}{21:26 $|$ 4:00:33}\\[2pt]
\mc{1}{|c|}{300000} &\multirow{1}{*}{2.730-5} &\multirow{1}{*}{6.826-6} &\multirow{1}{*}{513}&\multirow{1}{*}{ 9.6-7 $|$  5.5+0} &\multirow{1}{*}{2.0436+1 $|$  2.4743+1} &\multirow{1}{*}{59(306) $|$ 130}  &\multirow{1}{*}{47:05 $|$ 4:01:33}\\[2pt]
\hline
\end{tabular}
\end{center}
\end{table}

\subsection{Numerical results on the UCI data problems}
\label{subsec:Numerical results on the UCI data problems}

In this subsection, we also compare the the numerical performances between the SSN-ALM and the semi-proximal ADMM on the UCI data problems. We focus on the problems \eqref{eq:reparameterization} and \eqref{eq:sum-to-zero constraints}. We list the numerical results in Tables \ref{table-problem2-S1}-\ref{table-problem3-S2}. From the results, we can see that the SSN-ALM outperforms the semi-proximal ADMM by a large margin for the UCI data problems. Specifically, we can obtain highly accurate solutions efficiently by the SSN-ALM for all the problems. In contrast, the semi-proximal ADMM can hardly supply solutions that meet the accuracy requirement within the maximum time for nearly all the problems.

\begin{table}[!htbp]
\scriptsize
\begin{center}
\setlength{\belowcaptionskip}{10pt}
\parbox{.85\textwidth}{\caption{The performances of the SSN-ALM and semi-proximal ADMM on the synthetic datasets for the problem \eqref{eq:reparameterization} with the parameter setting ($S_{1}$). In this table, ``$a$''=SSN-ALM, ``$b$"=semi-proximal ADMM.}\label{table-problem2-S1}}
\begin{tabular}{|p{2.5cm}<{\centering}|p{0.8cm}<{\centering}|p{0.8cm}<{\centering}|p{0.4cm}<{\centering}|p{1.6cm}<{\centering}|p{2.8cm}<{\centering}|p{1.9cm}<{\centering}|p{1.7cm}<{\centering}|}
\hline
\multirow{1}{*}{pbname}&\multirow{3}{*}{$\lambda_{1}$}&\multirow{3}{*}{$\lambda_{2}$}&\multirow{3}{*}{nnz}&\multicolumn{1}{c|}{$\eta_{kkt}$}&\multicolumn{1}{c|}{$\pobj$}&\multicolumn{1}{c|}{iter}&\multicolumn{1}{c|}{time}\\
\cline{5-8} $(m,n,m_{E},m_{I})$ & & & & \multirow{2}{*}{$a\ |\ b$} & \multirow{2}{*}{$a\ |\ b$} & \multirow{2}{*}{$a\ |\ b$} &\multirow{2}{*}{$a\ |\ b$}\\[2pt]
$J$ & & & & & & & \\[2pt]
\hline
\mc{1}{|c|}{E2006.train}
&\multirow{1}{*}{1.940-1} &\multirow{1}{*}{1.940-1} &\multirow{1}{*}{20}&\multirow{1}{*}{ 4.6-7 $|$ 1.2-2} &\multirow{1}{*}{ 4.5613+3 $|$  4.4467+3} &\multirow{1}{*}{10(19) $|$ 30}  &\multirow{1}{*}{1:26 $|$ 4:05:15}
\\[2pt]
\mc{1}{|c|}{(16087,150360,48,0);} &\multirow{1}{*}{1.940-2} &\multirow{1}{*}{1.940-2} &\multirow{1}{*}{3} &\multirow{1}{*}{ 2.3-7 $|$  6.3-3} &\multirow{1}{*}{ 4.8716+2 $|$  4.8805+2} &\multirow{1}{*}{13(39) $|$ 32} &\multirow{1}{*}{2:47 $|$ 4:06:38}\\[2pt]
\mc{1}{|c|}{7518} &\multirow{1}{*}{1.940-3} &\multirow{1}{*}{1.940-3} &\multirow{1}{*}{89}&\multirow{1}{*}{ 6.6-7 $|$  5.8-3} &\multirow{1}{*}{9.1671+1 $|$  9.2247+1} &\multirow{1}{*}{17(78) $|$ 30}  &\multirow{1}{*}{10:41 $|$ 4:03:00}\\[2pt]
\hline
\mc{1}{|c|}{triazines.scale.expanded}
&\multirow{1}{*}{6.062-2} &\multirow{1}{*}{6.062-2} &\multirow{1}{*}{737}&\multirow{1}{*}{ 9.2-7 $|$ 6.6-3} &\multirow{1}{*}{ 6.9373-1 $|$  6.9651-1} &\multirow{1}{*}{28(157) $|$ 433}  &\multirow{1}{*}{4:16 $|$ 4:00:11}
\\[2pt]
\mc{1}{|c|}{(186,635376,48,0);} &\multirow{1}{*}{6.062-3} &\multirow{1}{*}{6.062-3} &\multirow{1}{*}{1013} &\multirow{1}{*}{ 8.9-7 $|$  6.3-3} &\multirow{1}{*}{ 2.9501-1 $|$  2.9610-1} &\multirow{1}{*}{43(195) $|$ 412} &\multirow{1}{*}{4:48 $|$ 4:00:06}\\[2pt]
\mc{1}{|c|}{31769} &\multirow{1}{*}{6.062-4} &\multirow{1}{*}{6.062-4} &\multirow{1}{*}{1024}&\multirow{1}{*}{ 8.3-7 $|$  2.8-3} &\multirow{1}{*}{2.3169-1 $|$  2.5109-1} &\multirow{1}{*}{51(327) $|$ 364}  &\multirow{1}{*}{7:07 $|$ 4:00:28}\\[2pt]
\hline
\mc{1}{|c|}{pyrim.scale.expanded}
&\multirow{1}{*}{2.440-4} &\multirow{1}{*}{2.440-4} &\multirow{1}{*}{424}&\multirow{1}{*}{ 8.8-7 $|$ 4.0-5} &\multirow{1}{*}{ 6.9227-1 $|$  3.7804-1} &\multirow{1}{*}{28(294) $|$ 8491}  &\multirow{1}{*}{56 $|$ 4:00:01}
\\[2pt]
\mc{1}{|c|}{(74,201376,48,0);} &\multirow{1}{*}{2.440-5} &\multirow{1}{*}{2.440-5} &\multirow{1}{*}{568} &\multirow{1}{*}{ 9.4-7 $|$  6.5-6} &\multirow{1}{*}{ 1.2228-1 $|$  7.9402-2} &\multirow{1}{*}{32(253) $|$ 9909} &\multirow{1}{*}{56 $|$ 4:00:00}\\[2pt]
\mc{1}{|c|}{10069} &\multirow{1}{*}{2.440-6} &\multirow{1}{*}{2.440-6} &\multirow{1}{*}{787}&\multirow{1}{*}{ 8.8-7 $|$  4.8-6} &\multirow{1}{*}{6.6985-2 $|$  6.5592-2} &\multirow{1}{*}{23(143) $|$ 8251}  &\multirow{1}{*}{1:02 $|$ 4:00:00}\\[2pt]
\hline
\mc{1}{|c|}{housing.scale.expanded}
&\multirow{1}{*}{5.701+0} &\multirow{1}{*}{5.701+0} &\multirow{1}{*}{27}&\multirow{1}{*}{ 4.9-7 $|$ 2.2-2} &\multirow{1}{*}{ 4.7022+4 $|$  4.6289+4} &\multirow{1}{*}{10(67) $|$ 238}  &\multirow{1}{*}{36 $|$ 4:00:08}
\\[2pt]
\mc{1}{|c|}{(506,77520,48,0);} &\multirow{1}{*}{5.701-1} &\multirow{1}{*}{5.701-1} &\multirow{1}{*}{159} &\multirow{1}{*}{ 7.0-7 $|$  1.0-2} &\multirow{1}{*}{ 3.4639+3 $|$  4.6442+3} &\multirow{1}{*}{12(86) $|$ 282} &\multirow{1}{*}{51 $|$ 4:00:34}\\[2pt]
\mc{1}{|c|}{3876} &\multirow{1}{*}{5.701-2} &\multirow{1}{*}{5.701-2} &\multirow{1}{*}{527}&\multirow{1}{*}{ 4.7-7 $|$  1.0-3} &\multirow{1}{*}{5.3645+1 $|$  2.4724+2} &\multirow{1}{*}{15(96) $|$ 305}  &\multirow{1}{*}{1:07 $|$ 4:00:23}\\[2pt]
\hline
\end{tabular}
\end{center}
\end{table}

\begin{table}[!htbp]
\scriptsize
\begin{center}
\setlength{\belowcaptionskip}{10pt}
\parbox{.85\textwidth}{\caption{The performances of the SSN-ALM and semi-proximal ADMM on the synthetic datasets for the problem \eqref{eq:reparameterization} with the parameter setting ($S_{2}$). In this table, ``$a$''=SSN-ALM, ``$b$"=semi-proximal ADMM.}\label{table-problem2-S2}}
\begin{tabular}{|p{2.5cm}<{\centering}|p{0.8cm}<{\centering}|p{0.8cm}<{\centering}|p{0.4cm}<{\centering}|p{1.6cm}<{\centering}|p{2.8cm}<{\centering}|p{1.9cm}<{\centering}|p{1.7cm}<{\centering}|}
\hline
\multirow{1}{*}{pbname}&\multirow{3}{*}{$\lambda_{1}$}&\multirow{3}{*}{$\lambda_{2}$}&\multirow{3}{*}{nnz}&\multicolumn{1}{c|}{$\eta_{kkt}$}&\multicolumn{1}{c|}{$\pobj$}&\multicolumn{1}{c|}{iter}&\multicolumn{1}{c|}{time}\\
\cline{5-8} $(m,n,m_{E},m_{I})$ & & & & \multirow{2}{*}{$a\ |\ b$} & \multirow{2}{*}{$a\ |\ b$} & \multirow{2}{*}{$a\ |\ b$} &\multirow{2}{*}{$a\ |\ b$}\\[2pt]
$J$ & & & & & & & \\[2pt]
\hline
\mc{1}{|c|}{E2006.train}
&\multirow{1}{*}{3.103-1} &\multirow{1}{*}{7.759-2} &\multirow{1}{*}{20}&\multirow{1}{*}{ 4.9-7 $|$ 1.2-2} &\multirow{1}{*}{ 7.2676+3 $|$  7.08461+3} &\multirow{1}{*}{10(20) $|$ 30}  &\multirow{1}{*}{1:47 $|$ 4:05:41}
\\[2pt]
\mc{1}{|c|}{(16087,150360,48,0);} &\multirow{1}{*}{3.103-2} &\multirow{1}{*}{7.759-3} &\multirow{1}{*}{3} &\multirow{1}{*}{ 7.3-7 $|$  1.6-3} &\multirow{1}{*}{ 7.4997+2 $|$  7.5035+2} &\multirow{1}{*}{12(39) $|$ 30} &\multirow{1}{*}{3:07 $|$ 4:04:14}\\[2pt]
\mc{1}{|c|}{7518} &\multirow{1}{*}{3.103-3} &\multirow{1}{*}{7.759-4} &\multirow{1}{*}{600}&\multirow{1}{*}{ 9.2-7 $|$  6.1-3} &\multirow{1}{*}{1.2102+2 $|$  1.1832+2} &\multirow{1}{*}{17(99) $|$ 27}  &\multirow{1}{*}{16:02 $|$ 4:02:20}\\[2pt]
\hline
\mc{1}{|c|}{triazines.scale.expanded}
&\multirow{1}{*}{9.700-2} &\multirow{1}{*}{2.425-2} &\multirow{1}{*}{1813}&\multirow{1}{*}{ 7.9-7 $|$ 2.8-3} &\multirow{1}{*}{ 4.9387-1 $|$  4.9432-1} &\multirow{1}{*}{33(174) $|$ 457}  &\multirow{1}{*}{4:24 $|$ 4:00:17}
\\[2pt]
\mc{1}{|c|}{(186,635376,48,0);} &\multirow{1}{*}{9.700-3} &\multirow{1}{*}{2.425-3} &\multirow{1}{*}{1822} &\multirow{1}{*}{ 8.8-7 $|$  1.4-3} &\multirow{1}{*}{ 2.5726-1 $|$  2.5802-1} &\multirow{1}{*}{45(230) $|$ 422} &\multirow{1}{*}{6:06 $|$ 4:00:13}\\[2pt]
\mc{1}{|c|}{31769} &\multirow{1}{*}{9.700-4} &\multirow{1}{*}{2.425-4} &\multirow{1}{*}{1940}&\multirow{1}{*}{ 9.5-7 $|$  6.2-4} &\multirow{1}{*}{2.2746-1 $|$  2.4322-1} &\multirow{1}{*}{47(368) $|$ 449}  &\multirow{1}{*}{7:28 $|$ 4:00:28}\\[2pt]
\hline
\mc{1}{|c|}{pyrim.scale.expanded}
&\multirow{1}{*}{3.903-4} &\multirow{1}{*}{9.758-5} &\multirow{1}{*}{1211}&\multirow{1}{*}{ 3.9-7 $|$ 1.0-5} &\multirow{1}{*}{ 7.9688-1 $|$  6.8910-1} &\multirow{1}{*}{26(239) $|$ 9004}  &\multirow{1}{*}{50 $|$ 4:00:01}
\\[2pt]
\mc{1}{|c|}{(74,201376,48,0);} &\multirow{1}{*}{3.903-5} &\multirow{1}{*}{9.758-6} &\multirow{1}{*}{1229} &\multirow{1}{*}{ 8.7-7 $|$  8.2-6} &\multirow{1}{*}{ 1.2931-1 $|$  9.1925-2} &\multirow{1}{*}{23(179) $|$ 7885} &\multirow{1}{*}{52 $|$ 4:00:01}\\[2pt]
\mc{1}{|c|}{10069} &\multirow{1}{*}{3.903-6} &\multirow{1}{*}{9.758-7} &\multirow{1}{*}{1588}&\multirow{1}{*}{ 7.8-7 $|$  2.3-6} &\multirow{1}{*}{6.8215-2 $|$  6.6066-2} &\multirow{1}{*}{24(132) $|$ 9644}  &\multirow{1}{*}{1:25 $|$ 4:00:01}\\[2pt]
\hline
\mc{1}{|c|}{housing.scale.expanded}
&\multirow{1}{*}{9.121+0} &\multirow{1}{*}{2.280+0} &\multirow{1}{*}{25}&\multirow{1}{*}{ 4.5-7 $|$ 1.9-2} &\multirow{1}{*}{ 3.0438+5 $|$  9.2003+4} &\multirow{1}{*}{14(98) $|$ 267}  &\multirow{1}{*}{33 $|$ 4:00:45}
\\[2pt]
\mc{1}{|c|}{(506,77520,48,0);} &\multirow{1}{*}{9.121-1} &\multirow{1}{*}{2.280-1} &\multirow{1}{*}{362} &\multirow{1}{*}{ 3.0-7 $|$  5.1-3} &\multirow{1}{*}{ 1.1145+4 $|$  7.8453+3} &\multirow{1}{*}{12(71) $|$ 273} &\multirow{1}{*}{43 $|$ 4:00:13}\\[2pt]
\mc{1}{|c|}{3876} &\multirow{1}{*}{9.121-1} &\multirow{1}{*}{2.280-2} &\multirow{1}{*}{1021}&\multirow{1}{*}{ 9.3-7 $|$  5.7-4} &\multirow{1}{*}{4.4121+1 $|$  4.4143+1} &\multirow{1}{*}{13(74) $|$ 293}  &\multirow{1}{*}{1:18 $|$ 4:00:33}\\[2pt]
\hline
\end{tabular}
\end{center}
\end{table}

\begin{table}[!htbp]
\scriptsize
\begin{center}
\setlength{\belowcaptionskip}{10pt}
\parbox{.85\textwidth}{\caption{The performances of the SSN-ALM and semi-proximal ADMM on the synthetic datasets for the problem \eqref{eq:reparameterization} with the parameter setting ($S_{1}$). In this table, ``$a$''=SSN-ALM, ``$b$"=semi-proximal ADMM.}\label{table-problem3-S1}}
\begin{tabular}{|p{2.5cm}<{\centering}|p{0.85cm}<{\centering}|p{0.85cm}<{\centering}|p{0.4cm}<{\centering}|p{1.6cm}<{\centering}|p{2.8cm}<{\centering}|p{1.9cm}<{\centering}|p{1.7cm}<{\centering}|}
\hline
\multirow{1}{*}{pbname}&\multirow{3}{*}{$\lambda_{1}$}&\multirow{3}{*}{$\lambda_{2}$}&\multirow{3}{*}{nnz}&\multicolumn{1}{c|}{$\eta_{kkt}$}&\multicolumn{1}{c|}{$\pobj$}&\multicolumn{1}{c|}{iter}&\multicolumn{1}{c|}{time}\\
\cline{5-8} $(m,n,m_{E},m_{I})$ & & & & \multirow{2}{*}{$a\ |\ b$} & \multirow{2}{*}{$a\ |\ b$} & \multirow{2}{*}{$a\ |\ b$} &\multirow{2}{*}{$a\ |\ b$}\\[2pt]
$J$ & & & & & & & \\[2pt]
\hline\mc{1}{|c|}{housing.scale.expanded}
	 &\multirow{1}{*}{5.701+0} &\multirow{1}{*}{5.701+0} &\multirow{1}{*}{14}&\multirow{1}{*}{5.7-7 $|$  1.4-3} &\multirow{1}{*}{ 3.7574+2 $|$  3.2811+2} &\multirow{1}{*}{11(48) $|$ 10000}  &\multirow{1}{*}{27 $|$ 1:53:57}\\[2pt]
\mc{1}{|c|}{(506,77520,1,0);} &\multirow{1}{*}{5.701-1} &\multirow{1}{*}{5.701-1} &\multirow{1}{*}{68} &\multirow{1}{*}{ 4.4-7 $|$  1.6-6} &\multirow{1}{*}{ 1.6595+5 $|$  1.6578+5} &\multirow{1}{*}{12(58) $|$ 10000} &\multirow{1}{*}{38 $|$ 2:10:06}\\[2pt]
\mc{1}{|c|}{7752} &\multirow{1}{*}{5.701-2} &\multirow{1}{*}{5.701-2} &\multirow{1}{*}{429}&\multirow{1}{*}{ 4.8-7 $|$  9.9-7} &\multirow{1}{*}{1.3394+4 $|$  1.3394+4} &\multirow{1}{*}{12(55) $|$ 2956}  &\multirow{1}{*}{33 $|$ 26:11}\\[2pt]
\hline
\mc{1}{|c|}{E2006.test}
&\multirow{1}{*}{4.733-2} &\multirow{1}{*}{4.733-2} &\multirow{1}{*}{10}&\multirow{1}{*}{ 8.3-7 $|$ 4.1-3} &\multirow{1}{*}{ 1.2016+4 $|$  1.2211+4} &\multirow{1}{*}{4(17) $|$ 1124}  &\multirow{1}{*}{1:32 $|$ 4:00:09}
\\[2pt]
\mc{1}{|c|}{(3308,150358,1,0);} &\multirow{1}{*}{4.733-3} &\multirow{1}{*}{4.733-3} &\multirow{1}{*}{10} &\multirow{1}{*}{ 3.8-7 $|$  4.8-4} &\multirow{1}{*}{ 1.2219+3 $|$  1.2192+3} &\multirow{1}{*}{10(20) $|$ 934} &\multirow{1}{*}{59 $|$ 4:00:08}\\[2pt]
\mc{1}{|c|}{15036} &\multirow{1}{*}{4.733-4} &\multirow{1}{*}{4.733-4} &\multirow{1}{*}{695}&\multirow{1}{*}{ 6.8-7 $|$  9.0-4} &\multirow{1}{*}{1.3575+2 $|$  1.3618+2} &\multirow{1}{*}{19(104) $|$ 1144}  &\multirow{1}{*}{2:33 $|$ 4:00:02}\\[2pt]
\hline
\mc{1}{|c|}{pyrim.scale.expanded}
&\multirow{1}{*}{2.400-4} &\multirow{1}{*}{2.400-4} &\multirow{1}{*}{474}&\multirow{1}{*}{ 4.2-7 $|$ 4.7-5} &\multirow{1}{*}{ 5.2631+0 $|$  4.3656+0} &\multirow{1}{*}{16(87) $|$ 10000}  &\multirow{1}{*}{29 $|$ 2:47:58}
\\[2pt]
\mc{1}{|c|}{(74,201376,1,0);} &\multirow{1}{*}{2.400-5} &\multirow{1}{*}{2.400-5} &\multirow{1}{*}{476} &\multirow{1}{*}{ 9.9-7 $|$  1.3-5} &\multirow{1}{*}{ 5.7609-1 $|$  2.1273-1} &\multirow{1}{*}{16(80) $|$ 10000} &\multirow{1}{*}{28 $|$ 2:40:48}\\[2pt]
\mc{1}{|c|}{20138} &\multirow{1}{*}{2.400-6} &\multirow{1}{*}{2.400-6} &\multirow{1}{*}{534}&\multirow{1}{*}{ 5.3-7 $|$  4.1-6} &\multirow{1}{*}{2.1273-1 $|$  7.0729-2} &\multirow{1}{*}{14(54) $|$ 10000}  &\multirow{1}{*}{28 $|$ 2:48:43}\\[2pt]
\hline
\mc{1}{|c|}{triazines.scale.expanded}
&\multirow{1}{*}{6.062-1} &\multirow{1}{*}{6.062-1} &\multirow{1}{*}{246}&\multirow{1}{*}{ 1.7-7 $|$ 8.3-2} &\multirow{1}{*}{ 2.2854+0 $|$  2.2356+0} &\multirow{1}{*}{7(56) $|$ 611}  &\multirow{1}{*}{6:28 $|$ 4:00:10}
\\[2pt]
\mc{1}{|c|}{(186,635376,1,0);} &\multirow{1}{*}{6.062-2} &\multirow{1}{*}{6.062-2} &\multirow{1}{*}{981} &\multirow{1}{*}{ 8.1-8 $|$  3.8-3} &\multirow{1}{*}{ 7.6679-1 $|$  7.6780-1} &\multirow{1}{*}{7(88) $|$ 661} &\multirow{1}{*}{7:10 $|$ 4:00:02}\\[2pt]
\mc{1}{|c|}{63538} &\multirow{1}{*}{6.062-3} &\multirow{1}{*}{6.062-3} &\multirow{1}{*}{1152}&\multirow{1}{*}{ 5.0-7 $|$  1.7-3} &\multirow{1}{*}{3.0670-1 $|$  3.0847-1} &\multirow{1}{*}{12(60) $|$ 566}  &\multirow{1}{*}{4:52 $|$ 4:00:24}\\[2pt]
\hline
\end{tabular}
\end{center}
\end{table}

\begin{table}[!htbp]
\scriptsize
\begin{center}
\setlength{\belowcaptionskip}{10pt}
\parbox{.85\textwidth}{\caption{The performances of the SSN-ALM and semi-proximal ADMM on the synthetic datasets for the problem \eqref{eq:reparameterization} with the parameter setting ($S_{2}$). In this table, ``$a$''=SSN-ALM, ``$b$"=semi-proximal ADMM.}\label{table-problem3-S2}}
\begin{tabular}{|p{2.5cm}<{\centering}|p{0.85cm}<{\centering}|p{0.85cm}<{\centering}|p{0.4cm}<{\centering}|p{1.6cm}<{\centering}|p{2.8cm}<{\centering}|p{1.9cm}<{\centering}|p{1.7cm}<{\centering}|}
\hline
\multirow{1}{*}{pbname}&\multirow{3}{*}{$\lambda_{1}$}&\multirow{3}{*}{$\lambda_{2}$}&\multirow{3}{*}{nnz}&\multicolumn{1}{c|}{$\eta_{kkt}$}&\multicolumn{1}{c|}{$\pobj$}&\multicolumn{1}{c|}{iter}&\multicolumn{1}{c|}{time}\\
\cline{5-8} $(m,n,m_{E},m_{I})$ & & & & \multirow{2}{*}{$a\ |\ b$} & \multirow{2}{*}{$a\ |\ b$} & \multirow{2}{*}{$a\ |\ b$} &\multirow{2}{*}{$a\ |\ b$}\\[2pt]
$J$ & & & & & & & \\[2pt]
\hline\mc{1}{|c|}{housing.scale.expanded}
	 &\multirow{1}{*}{9.121+0} &\multirow{1}{*}{2.280+0} &\multirow{1}{*}{19}&\multirow{1}{*}{2.0-7 $|$  8.9-4} &\multirow{1}{*}{ 3.6539+2 $|$  2.3640+2} &\multirow{1}{*}{13(54) $|$ 10000}  &\multirow{1}{*}{36 $|$ 2:00:29}\\[2pt]
\mc{1}{|c|}{(506,77520,1,0);} &\multirow{1}{*}{9.121-1} &\multirow{1}{*}{2.280-1} &\multirow{1}{*}{73} &\multirow{1}{*}{ 3.2-7 $|$  2.5-6} &\multirow{1}{*}{ 2.2607+5 $|$  2.2610+5} &\multirow{1}{*}{12(56) $|$ 10000} &\multirow{1}{*}{38 $|$ 2:45:36}\\[2pt]
\mc{1}{|c|}{7752} &\multirow{1}{*}{9.121-2} &\multirow{1}{*}{2.280-2} &\multirow{1}{*}{568}&\multirow{1}{*}{ 5.2-7 $|$  9.9-7} &\multirow{1}{*}{1.8789+4 $|$  1.8789+4} &\multirow{1}{*}{14(56) $|$ 3411}  &\multirow{1}{*}{32 $|$ 28:22}\\[2pt]
\hline
\mc{1}{|c|}{E2006.test}
&\multirow{1}{*}{7.570-2} &\multirow{1}{*}{1.890-2} &\multirow{1}{*}{10}&\multirow{1}{*}{ 3.6-7 $|$ 8.6-3} &\multirow{1}{*}{ 1.9210+4 $|$  1.9210+4} &\multirow{1}{*}{6(10) $|$ 1116}  &\multirow{1}{*}{43 $|$ 4:00:07}
\\[2pt]
\mc{1}{|c|}{(3308,150358,1,0);} &\multirow{1}{*}{7.570-3} &\multirow{1}{*}{1.890-3} &\multirow{1}{*}{10} &\multirow{1}{*}{ 7.6-7 $|$  1.1-3} &\multirow{1}{*}{ 1.9419+3 $|$  1.9482+3} &\multirow{1}{*}{9(17) $|$ 1108} &\multirow{1}{*}{1:02 $|$ 4:00:08}\\[2pt]
\mc{1}{|c|}{15036} &\multirow{1}{*}{7.570-4} &\multirow{1}{*}{1.890-4} &\multirow{1}{*}{405}&\multirow{1}{*}{ 3.4-7 $|$  8.8-4} &\multirow{1}{*}{2.0948+2 $|$  2.0952+2} &\multirow{1}{*}{19(97) $|$ 983}  &\multirow{1}{*}{2:07 $|$ 4:00:13}\\[2pt]
\hline
\mc{1}{|c|}{pyrim.scale.expanded}
&\multirow{1}{*}{3.903-4} &\multirow{1}{*}{9.758-5} &\multirow{1}{*}{539}&\multirow{1}{*}{ 5.0-7 $|$ 6.2-5} &\multirow{1}{*}{ 7.4039+0 $|$  6.7661+0} &\multirow{1}{*}{15(89) $|$ 10000}  &\multirow{1}{*}{33 $|$ 3:01:02}
\\[2pt]
\mc{1}{|c|}{(74,201376,1,0);} &\multirow{1}{*}{3.903-5} &\multirow{1}{*}{9.758-6} &\multirow{1}{*}{535} &\multirow{1}{*}{ 7.3-7 $|$  1.3-5} &\multirow{1}{*}{ 7.9365-1 $|$  3.8679-1} &\multirow{1}{*}{15(68) $|$ 10000} &\multirow{1}{*}{28 $|$ 2:46:37}\\[2pt]
\mc{1}{|c|}{20138} &\multirow{1}{*}{3.903-6} &\multirow{1}{*}{9.758-7} &\multirow{1}{*}{617}&\multirow{1}{*}{ 4.9-7 $|$  7.1-6} &\multirow{1}{*}{1.1597-1 $|$  7.6044-2} &\multirow{1}{*}{14(51) $|$ 10000}  &\multirow{1}{*}{29 $|$ 3:02:51}\\[2pt]
\hline
\mc{1}{|c|}{triazines.scale.expanded}
&\multirow{1}{*}{9.700-1} &\multirow{1}{*}{2.425-1} &\multirow{1}{*}{259}&\multirow{1}{*}{ 7.1-7 $|$ 6.2-2} &\multirow{1}{*}{ 1.9711+0 $|$  1.7830+0} &\multirow{1}{*}{9(58) $|$ 646}  &\multirow{1}{*}{6:23 $|$ 4:00:10}
\\[2pt]
\mc{1}{|c|}{(186,635376,1,0);} &\multirow{1}{*}{9.700-2} &\multirow{1}{*}{2.425-2} &\multirow{1}{*}{1152} &\multirow{1}{*}{ 5.9-7 $|$  4.7-3} &\multirow{1}{*}{ 6.3766-1 $|$  6.4020-1} &\multirow{1}{*}{13(110) $|$ 657} &\multirow{1}{*}{8:50 $|$ 4:00:17}\\[2pt]
\mc{1}{|c|}{63538} &\multirow{1}{*}{9.700-3} &\multirow{1}{*}{2.425-3} &\multirow{1}{*}{1417}&\multirow{1}{*}{ 6.3-7 $|$  1.6-3} &\multirow{1}{*}{2.7556-1 $|$  2.7690-1} &\multirow{1}{*}{11(70) $|$ 667}  &\multirow{1}{*}{4:11 $|$ 4:00:03}\\[2pt]
\hline
\end{tabular}
\end{center}
\end{table}
\section{Conclusion}
\label{sec:Conclusion}

In this paper, we have developed a dual SSN based ALM for the large-scale linearly constrained sparse group square-root Lasso problems. In this process, we mainly overcome the difficulty of dealing with two nonsmooth terms. We have established an equivalent condition to characterize the nonsingularity of the generalized Jacobian. And we have also presented the convergence theory of the algorithm. Finally, we have presented the numerical results to demonstrate the efficiency of the proposed algorithm.



\end{document}